\documentclass[11pt]{amsart}
\usepackage[curve,all]{xy}
 \xyoption{curve}
 \xyoption{line}
\xyoption{arc}
\usepackage{graphicx}
\usepackage{ amssymb,amsmath,amscd,xy}
\usepackage[dvips]{epsfig}

\usepackage{xy}
\xyoption{matrix}

\usepackage{times}
\usepackage{amsthm}
\newtheorem{theorem}{Theorem}[section]

\newtheorem{lemma}[theorem]{Lemma}

\newtheorem{proposition}[theorem]{Proposition}

\newtheorem{corollary}[theorem]{Corollary}

\theoremstyle{definition}     

\newtheorem{definition}[theorem]{Definition}

\newtheorem{example}[theorem]{Example}

\newenvironment{thm-A}
{{\vspace{0.1cm} \noindent \bf Main Theorem$\,$}\it}{\vspace{0.1cm}}

\newenvironment{thm-B}
{{\vspace{0.1cm} \noindent \bf Theorem B.$\,$}\it}{\vspace{0.1cm}}

\theoremstyle{remark}

\newtheorem{remark}[theorem]{Remark}

\newcommand{\calA}{\mathcal{A}}

\newcommand{\calE}{\mathcal{E}}

\newcommand{\frakm}{\mathfrak{m}}

\newcommand{\frako}{\mathfrak{o}}

\newcommand{\calL}{\mathcal{L}}

\newcommand{\calN}{\mathcal{N}}
\newcommand{\calO}{\mathcal{O}}
\newcommand{\calP}{\mathcal{P}}
\newcommand{\calQ}{\mathcal{Q}}
\newcommand{\calR}{\mathcal{R}}

\newcommand{\be}{\mathcal{\begin{equation}}}
\newcommand{\ee}{\mathcal{\end{equation}}}

\newcommand{\bbE}{\mathbb{E}}

\newcommand{\bbH}{\mathbb{H}}

\newcommand{\bbK}{\mathbb{K}}

\newcommand{\bbF}{\mathbb{F}}
\newcommand{\bbC}{\mathbb{C}}

\newcommand{\bbP}{\mathbb{P}}
\newcommand{\bbQ}{\mathbb{Q}}
\newcommand{\bbR}{\mathbb{R}}

\newcommand{\bbZ}{\mathbb{Z}}

\newcommand{\bfe}{\mathbf{e}}

\newcommand{\frakd}{\mathfrak{d}}
\newcommand{\fraka}{\mathfrak{a}}

\newcommand{\frakh}{\mathfrak{h}}
\newcommand{\frakr}{\mathfrak{r}}

\newcommand{\bfF}{\mathbf{F}}

\newcommand{\Or}{\textup{O}}
\newcommand{\SL}{\textup{SL}}
\newcommand{\SO}{\textup{SO}}
\newcommand{\Spin}{\textup{Spin}}

\newcommand{\Spec}{\textup{Spec}}
\newcommand{\Pic}{\textup{Pic}}

\newcommand{\PGL}{\textup{PGL}}

\newcommand{\id}{\textup{id}}

\newcommand{\Jac}{\textup{Jac}}

\newcommand{\rank}{\textup{rank}}
\newcommand{\cha}{\textup{char}}

\newcommand{\la}{\langle}
\newcommand{\ra}{\rangle}
\newcommand{\half}{\tfrac{1}{2}}

\newcommand{\Aut}{\textup{Aut}}

\newcommand{\tr}{\textup{tr}}
\newcommand{\calHa}{{\sf{Halp}}}

\newcommand{\gen}{\textup{gen}}
\newcommand{\fib}{\textup{fib}}

\newcommand{\beq}{\begin{equation}}
\newcommand{\eeq}{\end{equation}}

\newcommand{\bfmu}{\boldsymbol{\mu}}
\newcommand{\bfSO}{\mathbf{SO}}
\newcommand{\bfSpin}{\mathbf{Spin}}

\newcommand{\Num}{\textup{Num}}

\newcommand{\bl}{\textup{b}}

\date{}

\title[Automorphisms and Coble surfaces]{Rational surfaces with a large group of automorphisms}
\author{Serge Cantat and Igor Dolgachev}
\address{IRMAR, UMR 6625 du CNRS et Universit\'e de Rennes 1, B\^at. 22-23 du Campus de Beaulieu, F-35042 Rennes cedex; DMA, UMR 8553 du CNRS, \'Ecole Normale Sup\'erieure de Paris, 45 rue d'Ulm, F-75230 Paris cedex 05}
\address{Department of Mathematics, University of Michigan, 525 E. University Av., Ann Arbor, Mi, 49109}
\email{serge.cantat@univ-rennes1.fr}
\email{idolga@umich.edu}

\begin{document}
\maketitle

\begin{abstract} 
We classify rational surfaces $X$ for which the image of the automorphisms
group $\Aut(X)$ in the group of linear transformations of  the Picard group $\Pic(X)$ 
is the largest possible. This answers a question 
raised by Arthur Coble in 1928, and can be rephrased in terms of periodic 
orbits of  birational actions of  infinite Coxeter groups. 
 \end{abstract}

\setcounter{tocdepth}{1}
\tableofcontents

\section{Introduction} 

Let $\bbK$ be an algebraically closed field, and $X$ be a projective surface defined over $\bbK$.
The group of automorphisms $\Aut(X)$ acts  on the N\'eron-Severi group   of $X$. This action 
preserves the intersection form and the canonical class $K_X$, and therefore provides a morphism from 
$\Aut(X)$ to the group of integral isometries $\Or(K_X^\perp)$ of the orthogonal complement $K_X^\perp$. When $X$ is rational, the image
satisfies further constraints: It is contained in an explicit Coxeter subgroup $W_X$ of $\Or(K_X^\perp)$, and $W_X$ has infinite index in $\Or(K_X^\perp)$ as soon as the rank $\rho(X)$ of the N\'eron-Severi group of $X$ exceeds $11$.

A natural problem is to describe all projective surfaces 
$X$ for which $\Aut(X)$ is infinite  and its image in this orthogonal or Coxeter group is of finite index. 
When $\bbK$ is the field of complex numbers, the problem asks for a classification of 
 complex projective surfaces with maximal possible groups of isotopy classes of holomorphic 
diffeomorphisms.

We solve this problem  when $X$ is a rational surface. It turns out that this is the most interesting and difficult case. In Section \ref{par:NRS} we briefly discuss other types of surfaces; one can treat them by more or less standard arguments. 

\subsection{Automorphisms of rational surfaces} Let $X$ be a rational surface  defined over~$\bbK$. The N\'eron-Severi group of $X$ coincides with the  Picard group $\Pic(X)$; its rank $\rho(X)$ is the Picard number of $X$.  
We denote by $\Aut(X)^*$ the image of $\Aut(X)$ in the orthogonal group $\Or(\Pic(X))$. 
There are two alternative possibilities for $\Aut(X)$ to be infinite. 

The first  occurs when
the kernel $\Aut(X)^0$ of the action of $\Aut(X)$ on $\Pic(X)$  is infinite. In this case, $\Aut(X)^0$ is a linear algebraic group of positive dimension and $\Aut(X)^*$ is a finite group (see \cite{Harbourne}).
All such examples are easy to describe because the surface $X$ is obtained from a minimal rational surface  by a sequence of $\Aut(X)^0$-equivariant  blow-ups. Toric surfaces provide examples of this kind.  

In the second case, the group $\Aut(X)^*$ is infinite, $\Aut(X)^0$ is finite, and then  
$X$ is obtained from the projective plane $\bbP^2$ by blowing up a sequence of points 
$p_1, \ldots, p_n$, with $n\geq 9$ (see \cite{Nagata:II}). 
The existence of such an infinite  group $\Aut(X)^*$ imposes drastic constraints on the 
point set $\calP=\{p_1, \ldots, p_n\}$ and leads to nice geometric properties of this set. 
There are classical examples of this kind as well as very recent constructions 
(see \cite{Dolgachev-Ortland:Ast}, \cite{Bedford-Kim:2006},
\cite{McMullen:2007},  \cite{Takenawa:2001}). Our goal is to classify point sets $\calP$ for which
the group $\Aut(X)^*$ is the largest possible, in a sense which we now make more precise. 

\subsection{The hyperbolic lattice}\label{par:Minkowski-Lattice} Let $\bbZ^{1,n}$ denote the standard odd unimodular lattice of  signature~$(1,n)$.
It is generated by an orthogonal basis $(\bfe_0, \bfe_1, \ldots, \bfe_n)$ with
\[
\bfe_0^2=1, \quad {\text{and}}\quad \bfe_i^2=-1 \quad {\text{for}}\quad i\geq 1.
\]
The orthogonal complement of the vector  
\[
k_n = -3\bfe_0+(\bfe_1+\cdots +\bfe_n)
\]
is a sublattice $\bbE_n\subset \bbZ^{1,n}$. A basis of $\bbE_n$ is formed by the vectors 
\[
\boldsymbol{\alpha}_0 = \bfe_0-\bfe_1-\bfe_2-\bfe_3, \; {\text{and}} \;\, \boldsymbol{\alpha}_i = \bfe_i-\bfe_{i+1}, \ i = 1,\ldots, n-1.
\]
The intersection matrix  $(\boldsymbol{\alpha}_i\cdot \boldsymbol{\alpha}_j)$ is equal to $ \Gamma_{n} - 2I_{n}$, where $\Gamma_{n}$ is the incidence matrix of the graph
$T_{2,3,n-3}$ from Figure \ref{fig:Dynkin}. In particular, each class $\boldsymbol{\alpha}_i$ has self-intersection $-2$ and
determines an involutive isometry of $\bbZ^{1,n}$ by
\[
\boldsymbol{s}_i:x\mapsto x+ (x\cdot \boldsymbol{\alpha}_i) \boldsymbol{\alpha}_i.
\]
By definition, these involutions generate the Coxeter (or Weyl)  group $W_n$. 

\begin{figure}[ht]
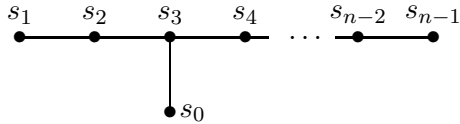

\xy
(-35,0)*{};
@={(0,0),(10,0),(20,0),(30,0),(45,0),(55,0),(20,-10)}@@{*{\bullet}};
(0,0)*{};(33,0)*{}**\dir{-};(42,0)*{};(55,0)*{}**\dir{-};
(20,0)*{};(20,-10)*{}**\dir{-};(38,0)*{\ldots};
(0,3)*{s_1}; (10,3)*{s_{2}}; (20,3)*{s_3};(30,3)*{s_4}; (45,3)*{s_{n-2}};(55,3)*{s_{n-1}};(23,-10)*{s_0};
\endxy
\caption{{\sf{Coxeter-Dynkin diagram of type $T_{2,3,n-3}$}}}\label{fig:Dynkin}
\end{figure}

\subsection{Automorphisms and Coxeter groups}\label{par:WXWn}

>From now on, $X$ is a rational surface for which $\Aut(X)^*$ is infinite. We write $X$ as the blow-up of $\bbP^2$ at $n$ points 
$p_1$, ..., $p_n$ with $n\geq 9$; some
of them can be infinitely near points and, by convention, $j\geq i$ if $p_j$ is infinitely near   $p_i$.  
We now describe known constraints on the structure of the group $\Aut(X)^*$. 

A basis $(e_0,\ldots,e_n)$ of $\Pic(X)$ is obtained by taking for $e_0$ the class of the total transform of a line in $\bbP^2$, and for $e_i$, $1\leq i\leq n,$ the class of the total transform of the exceptional divisor obtained by blowing up $p_i$; in particular, the 
Picard number $\rho(X)$  is equal to $n+1$.
This basis
is orthogonal with respect to the intersection form: $e_0^2=1$, $e_i^2=-1$ for $i\geq 1$, and $e_i\cdot e_j=0$ if $i\neq j$. We call such a basis of $\Pic(X)$ a {\bf geometric basis}. A geometric basis makes  $\Pic(X)$  isometric to the lattice
$\bbZ^{1,n}$, by an isometry which maps  $e_i$ to $\bfe_i$. Under this isomorphism, the canonical class 
\[
K_X=-3e_0+e_1+\cdots + e_n
\]
is mapped to the element  $k_n\in \bbZ^{1,n}$. We denote by $(\alpha_i)$ the basis of $K_X^\perp$ corresponding
to $(\boldsymbol{\alpha}_i)$ under this isomorphism, i.e. 
\beq\label{basis}
\alpha_0= e_0-e_1 -e_2 -e_3,\   \alpha_1 = e_1-e_2,\ \ldots, \ \alpha_{n-1} = e_{n-1} - e_{n},
\eeq
and by $s_i$ the involutive  isometry of $\Pic(X)$ which is conjugate to $\boldsymbol{s}_i$. 
By definition, the group $W_X$ is the group of isometries of $\Pic(X)$ generated by these $n$ involutions; 
thus, $W_X$ is isomorphic to the Coxeter group $W_n$. It is known that the group $W_X$ does not depend on a choice of a geometric basis (see \cite{Dolgachev:2008}, Theorem~5.2, page 27). 

The group $\Aut(X)^*$ acts by isometries on $\Pic(X)$ and preserves the canonical class.
According to Kantor-Nagata's theorem (see \cite{Kantor}, Theorem XXXIII, \cite{Nagata:II}, p. 283, or \cite{Dolgachev:2008}, Theorem 5.2), 
the group $\Aut(X)^*$ is contained in  $W_X$. Thus, we get a series of  inclusions 
\[
\Aut(X)^*\subset W_X \subset \Or(\Pic(X); K_X) \subset \Or(\Pic(X)),
\]
where  $\Or(\Pic(X))$ is the orthogonal group of $\Pic(X)$ with respect to the intersection form, and $\Or(\Pic(X); K_X)$ is the stabilizer of the canonical class $K_X$.

\subsection{Cremona special point sets}
When $n\leq 8$, $W_n$ is a finite group.  We say that  the point set $\calP:=\{p_1, ..., p_n\}$ 
is {\bf{Cremona special}} if $n\geq 9$ and the group $\Aut(X)^*$ has finite index in $W_X$.\footnote{ Coble called such subsets \emph{special}, so we somewhat deviate from his terminology.} 
A rational surface obtained by blowing up a Cremona special set will be called Cremona special. In this sense,   Cremona special surfaces with fixed Picard number are rational surfaces  with largest possible discrete automorphism groups among all rational surfaces with the same rank of the Picard group.  

Our goal is to classify Cremona special point sets, a problem that is already mentioned by Arthur Coble
in his book \cite{Coble:Book}, p. 278 (see also \cite{Dolgachev-Ortland:Ast} and  \cite{Harbourne}).

Two kinds of such sets are known since the beginning of the last century. They are general Halphen sets of $9$ points and general Coble sets of $10$ points (see \cite{Coble:Book}).  Brian Harbourne showed in \cite{Harbourne2} that,  in characteristic $p > 0$, and for any integer $n\geq 9$, a general set of $n$ nonsingular points on an irreducible cuspidal cubic curve is Cremona special; for this, he employed the fact that all such points are $p$-torsion points in the group law on the set of  nonsingular points on the cubic curve. 
When $n=9$, Harbourne sets are particular
cases of Halphen sets. 

We discuss the geometry of Halphen and Coble point sets  in Sections \ref{Halphen} and \ref{Coble}, prove that the general ones  are indeed Cremona special, and describe precisely what  ``general" means in this context. By definition, the point set is {\bf unnodal} if its blow-up  does not contain  smooth rational curves with self-intersection  equal to $-2$ (also called nodal or $(-2)$-curves). This terminology is borrowed from the theory of Enriques surfaces, where it is known that the isomorphism classes of unnodal surfaces form an open subset in the moduli space. A surface obtained by blowing up an unnodal set is called unnodal. We show that unnodal sets form an open Zariski  subset in the variety of point sets defining Halphen and Coble surfaces, and that Cremona special Halphen and Coble surfaces are exactly the unnodal ones. 

Harbourne examples are defined at the beginning of Section \ref{par:GC} and are discussed in Section \ref{par:Harbourne}. 

Then, our main result shows that the examples constructed by Halphen, Coble, and Harbourne 
exhaust all possibilities of Cremona special point sets. As a corollary, if a point set $\calP$ is Cremona special, then it is unnodal.

\begin{thm-A}\label{thm:Main}  
Let $\bbK$ be an algebraically closed field.
Let $\calP$ be a Cremona special point set in  $\bbP^2_\bbK$. Then $\calP$ is unnodal and one of the following cases occurs 
\begin{itemize}
\item  $n = 9$ and  $\calP$ is a  Halphen set;
\item $n = 10$ and $\calP$ is a  Coble set;
\item $n \ge 10$,  $\cha(\bbK) > 0$, and $\calP$ is a Harbourne set. 
\end{itemize}
Conversely, any such unnodal set is Cremona special.
\end{thm-A}

As a corollary, {\sl{if a rational surface $X$ is Cremona special, then $-K_X$ or $-2K_X$ is effective}}. 

\begin{remark}\label{Rem:Intro}a.-- As explained in Section \ref{Section:GeneralFields}, there is a stronger version of this theorem which does not 
assume that $\bbK$ is algebraically closed, but this requires a careful definition of  Cremona special point sets. Non rational surfaces are dealt with 
in Section \ref{par:NRS}.

1.1.$\,$ b.-- When $W_X$ is infinite, it is Zariski dense in the real algebraic group $\Or(K_X^\perp\otimes \bbR)$. Thus, a natural question is the following. If $X$ is a rational 
surface and $\Aut(X)^*$ is infinite and  Zariski dense in $\Or(K_X^\perp\otimes \bbR)$, does it follow that $X$ is Cremona special ? 
In other words, is it possible to generalize our Main Theorem under the weaker assumption that $\Aut(X)^*$ is  infinite and is Zariski 
dense in $\Or(K_X^\perp \otimes \bbR)$ ? Since $W_9$ contains a finite index, free abelian group of rank $8$, every Zariski dense
subgroup of $W_9$ has finite index. Thus, the problem concerns rational surfaces with Picard number at least $10$. 

1.1.$\,$ c.-- There is a notion of Cremona special point sets in projective spaces of higher dimension and in their Cartesian products. 
Interesting examples of such sets are known 
(see \cite{Coble:Book},  \cite{Dolgachev-Ortland:Ast}, \cite{Dolgachev:Hannover}). 
Unfortunately, the methods of this paper are specific to dimension $2$ and do not extend to the higher-dimensional case.
\end{remark}

\subsection{An action of $W_n$ on point sets and its periodic orbits}\label{1.5}
Consider the variety $(\bbP^2)^n$ and the diagonal action of $\PGL_3$ on it.   Consider the GIT-quotient  ${\sf{P}}^2_n$ of the action. It turns out that the group $W_n$ acts on ${\sf{P}}^2_n$ by birational transformations; this {\bf{Cremona action}}
is described in chapter VI of \cite{Dolgachev-Ortland:Ast}.

Let $\Gamma$ be a subgroup of $W_n$. Let $(p_1, \ldots, p_n)$ be an ordered stable set of distinct points representing a  point $\sf{p}\in {\sf{P}}^2_n$.
  Let $X_{\sf{p}}$
be the surface obtained by blowing up the projective plane at $p_1, \ldots, p_n$; its isomorphism class  depends only on $\sf{p}$. The group $\Pic(X_{\sf{p}})$ is isomorphic to $\bbZ^{1,n}$, with an isomorphism depending only on $\sf{p}$; 
 we fix such an isomorphism, and the corresponding isomorphism
between $W_X$ and $W_n$.

If ${\sf{p}}$ is contained in the domain of definition of   $\gamma$ and $\gamma({\sf{p}})={\sf{p}}$ for all $\gamma$
in $\Gamma$, then there exists a subgroup $\Gamma'\subset \Aut(X_{\sf{p}})$ such that the action of $\Gamma'$ 
on $\Pic(X_{\sf{p}})$ and the identification $W_n\sim W_{X_{\sf{p}}}$ provide an isomorphism $\Gamma' \to \Gamma$. In other 
words, {\sl{points ${\sf{p}}\in {\sf{P}}^2_n$ which are fixed by the group $\Gamma$ correspond to rational surfaces on which $\Gamma$ is
represented by a subgroup of $\Aut(X_{\sf{p}})$}}.

Thus,  our Main Theorem {\sl{classifies periodic orbits of the group $W_n$, for $n\geq 9$}}, i.e. for infinite Coxeter
groups $W_n$. This provides a dynamical interpretation of the Main Theorem in terms of birational actions
of Coxeter groups.

\begin{remark} There are other nice examples of algebraic dynamical systems for which periodic orbits are 
related to the construction of interesting geometric objects. 

One of them is given by Thurston's pull back map. If $F:S^2\to S^2$ 
is a (topological) orientation preserving branched covering map of the sphere $S^2$ with a finite post-critical set\footnote{The post-critical set is the union of the images of the set of critical points of $f$ under positive iterations of $f$.} $\calP_F$ of cardinality $n$,  
one can ask whether $F$ is equivalent to a holomorphic endomorphism
 $f$ of the Riemann sphere $\bbP^1(\bbC)$, in the sense that 
 \[
 F=\phi \circ f \circ \psi
 \]
 where $\phi$ and $\psi$ are homeomorphisms which are isotopic relatively to $\calP_F$. The   map $\sigma_F$ defined by Thurston acts on the Teichm\"uller
 space of $S^2$ with $n$-marked points;  fixed points of $\sigma_F$ correspond to holomorphic structures on the sphere
 for which $F$ is realized by an endomorphism $f$. This situation is similar to the one studied here, with $\sigma_F$ in
 place of $W_n$ and the Teichm\"uller space replacing ${\sf{P}}^2_n$. We refer to \cite{Douady-Hubbard} for a precise description 
 of Thurston's construction.
 
 Another similar situation, with the mapping class group of a surface $\Sigma$ (in place of $W_n$) acting on the character variety
 of the fundamental group $\pi_1(\Sigma)$ (in place of ${\sf{P}}^2_n$)  is related to hyperbolic structures on three dimensional manifolds, 
 and to algebraic solutions of Painlev\'e sixth equation (see \cite{Cantat:Duke} and references therein).
 \end{remark}
 
\subsection{Acknowledgement} Thanks to A. Chambert-Loir, J.-L. Colliot-Th\'el\`ene, Y. de Cornulier, M. W. Davis, M. H. Gizatullin, D. Harari, J. Keum, S. Kond\={o}, C.T. McMullen, V. Nikulin, and G. Prasad for interesting discussions on the topics of this paper. 
We thank the thorough referees for their careful reading and their comments which allowed us to clarify  the exposition and to correct some of the arguments of the previous version of this paper.


\section{Halphen surfaces}\label{Halphen}


In this section, we describe Halphen surfaces, Halphen pencils of genus $1$ curves, and
their associated point sets. We then show that unnodal Halphen point sets are Cremona
special. Most results in this section are known to experts, but may be hard to find in the literature, 
and will be used in the following sections.

We assume some familiarity with the theory of elliptic fibrations over fields of arbitrary characteristic and refer
to \cite{Dolgachev-Cossec:Book}, Chapter V, for this topic.

\subsection{Halphen surfaces of index m}\label{par:Halphen} 

By definition, a {\bf{ $(-n)$-curve}} on a smooth projective surface $X$ is a smooth rational curve with self-intersection $-n$. The genus formula shows the following. 

\begin{lemma}\label{-ncurves}
Let $X$ be a smooth projective surface. Let $n$ and $l$ be positive integers.
\begin{enumerate}
\item Assume $-K_X$ is nef. If $E$ is a $(-n)$-curve, 
then $n=1$, or $n=2$ and $E\cdot K_X=0$.
\item Assume that the linear system $\vert -lK_X\vert $ contains a reduced, irreducible curve $C$ with $C^2<0$. If
$E$ is a $(-n)$-curve, then $n=1$, or $n=2$ and $E\cdot C=0$, or $E=C$.
\end{enumerate}
\end{lemma}

A smooth rational projective surface $X$ is a {\bf{Halphen surface}} if there exists an integer $m > 0$ 
such that the linear system $| - m K_X|$ is of dimension $1$, has no fixed component, and has no base point.
The {\bf{index}} of a Halphen surface is the smallest possible value for such a  positive integer $m$.

Let $X$ be a Halphen surface of index $m$. Then $K_X^2 = 0$ and, by the genus formula, 
the linear system $| - m K_X|$ defines a genus $1$ fibration $f:X\to \bbP^1$, 
which is elliptic or quasi-elliptic if $\cha(\bbK) = 2$ or $3$.  
This fibration is {\bf{relatively minimal}} in the sense that there is  no $(-1)$-curve contained in a fiber.

\begin{proposition}\label{pro:halphen} Let $X$ be a smooth projective rational surface. Let $m$ be a positive integer. The following four properties are equivalent:
\begin{itemize}
\item[(i)] $X$ is a Halphen surface of index $m$;
\item [(ii)] $|-K_X|$ is nef and contains a  curve $F_0$ such that $\calO_{F_0}(F_0)$ is of order $m$ in $\Pic(F_0)$;
\item[(iii)] there exists a relatively minimal elliptic or  quasi-elliptic fibration $f:X\to \bbP^1$; it has  no multiple fibers when $m = 1$ and a unique  multiple fiber, of multiplicity $m$, when $m > 1$;
\item[(iv)] there exists an irreducible  pencil of curves of degree $3m$ with $9$ base points of multiplicity $m$ in $\bbP^2$, such that $X$ is the blow-up of the $9$ base points and $|-mK_X|$ is the proper transform of this pencil (the base point set may contain infinitely near points).
\end{itemize}
\end{proposition}

In the proof of (iii)$\Rightarrow $(iv) below,  the classification of minimal rational surfaces is used. Recall that a minimal rational surface
is isomorphic to $\bbP^2$ or to one of the Segre-Hirzebruch surfaces $\bfF_n = \bbP(\calO\oplus \calO(-n))$,\footnote{Here we adopt Grothendieck's definition of the projective bundle associated to a locally free sheaf.}  with $n\geq 0$ and $n\neq 1$. If $n=0$,  the
surface is isomorphic to $\bbP^1\times \bbP^1$ and, if $n=1$,  the surface is not minimal since it is isomorphic to the blow-up of $\bbP^2$ at one
point.  For all $n\geq 1$ there is a unique irreducible curve on $\bfF_n$ with negative self-intersection (equal to $-n$). It is defined by a section of the  $\bbP^1$-bundle $\bfF_n\to \bbP^1$ corresponding to  the surjection $\calO\oplus \calO(-n)\to \calO(-n)$ (\cite{Hartshorne:book}, \S V.2). 

\begin{proof}[Proof of Proposition \ref{pro:halphen}] Under assumption (i),  Riemann-Roch formula  on a rational surface and Serre's Duality
$$h^0(D)+h^0(K_X-D) = h^1(D)+\half D\cdot (D-K_X)+1$$
and $K_X^2=0$ imply that $h^0(-K_X) > 0$. Let $F_0$ be an element of the linear system $|-K_X|$. 

We now prove (i)$ \Leftrightarrow$(ii). The exact sequence
\beq\label{ex1}
0\to \calO_X\to \calO_X(nF_0)\to \calO_{F_0}(nF_0)\to 0
\eeq
together with  $h^1(X, \calO_X) = 0$, because $X$ is rational, show that 
\[
h^0(\calO_X(nF_0)) = 1+ h^0(\calO_{F_0}(nF_0)).
\]
Since $F_0$ is a nef divisor and $F_0^2=0$, the restriction of $\calO_X(nF_0)$ to each irreducible component of 
$F_0$ is an invertible sheaf of degree zero. The curve $F_0$ is of arithmetic genus 1, so we can apply 
the Riemann-Roch Theorem on $F_0$ (see \cite{Mumford}, Lecture 11)  to conclude that $h^0(\calO_{F_0}(nF_0)) > 0$ 
if and only if $\calO_{F_0}(nF_0) \cong \calO_{F_0}$, if and only if  $h^0(\calO_{F_0}(nF_0))= 1$ . 

This shows that the index $m$ can be characterized by the property
\[
m = \min\{n:h^0(\calO_{X}(-nK_X)) = 2\} = \min\{n:h^0(\calO_{F_0}(nF_0)) > 0\},
\]
and the equivalence (i)$ \Leftrightarrow$(ii) follows from this characterization. 
 
(i)$ \Rightarrow$ (iii) The pencil $|-mK_X|$ defines a morphism $f:X\to \bbP^1$ with general fiber of arithmetic genus $1$. The generic fiber $X_\eta$ is an irreducible curve of arithmetic genus $1$ over the field $\bbK(\eta)$ of rational functions on the curve $\bbP^1$. Since $X$ is smooth, $X_\eta$ is a regular curve over $\bbK(\eta)$. It is known that it is smooth if $\cha(\bbK) \ne 2,3$, so that in this case $f$ is an elliptic fibration (see \cite{Dolgachev-Cossec:Book}, Proposition 5.5.1). 
If it is not smooth, then a general fiber of $f$ is an irreducible cuspidal curve, so that $f$ is a quasi-elliptic fibration. As explained above (just before Proposition \ref{pro:halphen})  the fibration $f$ is relatively minimal. 

Since $X$ is a rational non-minimal surface, there exists a $(-1)$-curve $E$ on $X$ satisfying $E\cdot K_X = -1$. This shows that $K_X$ is a primitive divisor class, i.e. $K_X$ is not a multiple of any other divisor class. If $m = 1$, this implies that there are no multiple fibers. If $m > 1$, this implies that the multiplicity $n$ of any multiple fiber $nD$ divides $m$. Since $|mF_0|$ is a multiple fiber of multiplicity $m$,
the class of the divisor $\frac{m}{n}F_0-D$ is a torsion element in the Picard group  of $X$. Since $X$ is a rational surface, this class must be trivial, and we conclude that $f$ has a unique multiple fiber, namely $mF_0$.

(iii)$\Rightarrow $(iv) 
Since the fibration $f$ is relatively minimal, the canonical class is proportional to 
the class of the fibers of $f$ (see \cite{BPVDVH}, \S V.12); in particular, 
$K_X^2=0$ and $-K_X$ is nef.

Let $\pi:X\to Y$ be a birational morphism to a minimal ruled surface. Suppose  that $Y$ is not isomorphic to $\bbP^2$;  then $Y$ is isomorphic to  a  surface $\bfF_n$, $n \ne 1$.  
Let $E_0$ be the section of $Y$ with $E_0^2 = -n$ and let $E$ be its proper transform on $X$. We have  $E^2 \le -n$ and $E^2 = -n$ if and only if $\pi$ is an isomorphism in an open neighborhood of $E_0$. Since $-K_X$ is nef, Lemma \ref{-ncurves} shows that $n=0$ or $n=2$. 
Assume $n = 2$. 
Then $\pi$ is an isomorphism over $E_0$, hence it factors  through a birational  morphism $\pi:X\to X_1$, where $X_1$ is the blow-up of $\bfF_2$ at  a point $x\not\in E_0$. Let $X_1\to \bfF_{1}$ be the blow-down of the fiber of the ruling $\bfF_2\to \bbP^1$ passing through $x$. Then we obtain a birational morphism $X\to \bfF_1\to \bbP^2$. Assume now that $n = 0$. The morphism factors through $X\to X_2$, where $X_2$ is the blow-up of a point $y$ on $\bfF_0$. Then we compose $X\to X_2$ with the birational morphism $X_2\to \bbP^2$ which is given by the stereographic projection of $\bfF_0$ onto $\bbP^2$ from the point $y$. 
As a consequence, changing $Y$ and $\pi$, we can always assume that $Y=\bbP^2$. 

Since $K_X^2 = 0$, the morphism $\pi:X\to \bbP^2$ is the blow-up of $9$ points $p_1,\ldots,p_9$, where some of them may be infinitely near. Since $-K_X$ is nef, any smooth rational curve has   self-intersection $\ge -2$. This implies that the set of points $\{p_1,\ldots,p_9\}$ can be written in the form
\beq\label{set}
\{ p_{1}^{(1)},p_{1}^{(2)},\ldots,p_{1}^{(a_1)}; \ldots  ; p_k^{(1)},p_k^{(2)},\ldots, p_k^{(a_k)}\},
\eeq
where the $p_{j}^{(1)}$ are points in $\bbP^2$, and $p_j^{(b+1)}$ is infinitely near, of the first order, to the previous point $p_j^{(b)}$ for $j = 1,\ldots, k$ and $b = 1,\ldots,a_j-1$. Equivalently, the exceptional curve 
\[
E_j = \pi^{-1}(p_j^{(1)})
\] 
is a chain of $(-2)$-curves of length $(a_j-1)$ with one more $(-1)$-curve at the end of the chain.

The formula for the canonical class of the blow-up of a nonsingular surface at a closed point shows that 
\beq
K_X = -3e_0+e_1+\cdots +e_9,
\eeq
where $e_0 = c_1(\pi^*(\calO_{\bbP^2}(1))$ and $e_j$ is the divisor class of  $E_j,\  j = 1,\ldots,9.$ This implies that 
$$|-mK_X| = |3me_0-m(e_1+\cdots+e_9)|,$$
hence the image of the pencil $|-mK_X|$ in the plane is the linear system of curves of degree $3m$ with singular points of multiplicity $m$ at $p_i^{(1)}$, $1\leq j \leq k$. 

(iv)$\Rightarrow$ (i) Let $X$ be the blow-up of the base points of the pencil. The proper transform  of the pencil on $X$ is the linear system  $|3me_0-m(e_1+\cdots+e_9)|$. The formula for the canonical class on $X$ shows that this system is equal to $|-mK_X|$.  Since the pencil is irreducible, $|-mK_X|$ is a pencil with no fixed component and no base point, so $X$ is a Halphen surface.
\end{proof}

\begin{remark} The proof of the proposition shows that the multiplicity $m$ of the multiple fiber $mF_0$ of the genus one fibration is equal to the order of $\calO_{F_0}(F_0)$ in $\Pic(F_0)$. This property characterizes non-wild fibers of elliptic fibrations (see \cite{Dolgachev-Cossec:Book}, Proposition 5.1.5). It is a consequence of the vanishing of $H^1(X,\calO_X)$.  It always holds if the multiplicity is prime to the characteristic.
\end{remark}

\subsection{Halphen pencils of index m}\label{par:Halphen pencil} 
The following lemma is well-known and its proof is left to the reader.

\begin{lemma}\label{wk} Let $\phi:S'\to S$ be the blow-up of a point $x$ on a smooth projective surface $S$ and let $C'$ be the proper transform of a curve passing through $x$ with multiplicity $1$. Then 
$\calO_{C'}(C'+E) \cong (\phi _{|{C'}})^*\calO_C(C)$, where $E = \phi^{-1}(x)$
is the exceptional divisor.
\end{lemma}

In the plane $\bbP^2$, an irreducible pencil of elliptic curves of degree $3m$ with $9$ base points
of multiplicity $m$ is called a {\bf{Halphen pencil of index $\mathbf{m}$}}. If $C_0$ is a cubic curve through the base points, then $C_0$ is the image of a curve 
$F_0\in |-K_X|$; such a curve is unique if $m > 1$ and moves in the pencil if $m=1$. 

The classification of fibers of  genus $1$ fibrations shows that $\calO_{F_0}(F_0)\not\cong \calO_{F_0}$ implies that $F_0$ is a reduced divisor of type $I_m$ in Kodaira's notation, unless $\cha(\bbK)$ divides $m$ (see \cite{Dolgachev-Cossec:Book}, Proposition 5.1.8). We further assume that $F_0$ is irreducible if $m > 1$; this will be enough for our applications. Thus $F_0$ is a smooth or nodal curve, unless the characteristic of $\bbK$ divided $m$ in which case it could be a cuspidal curve.  Under this assumption, the restriction of $\pi$ to $F_0$ is an isomorphism $F_0\cong C_0$; in particular, no base point $p_i^{(j)}$ is a singular point of $C_0$. 

In the notation of Equation \eqref{set}, consider the divisor class in $\Pic(C_0)$ given by
\[
\frakd = 3\frakh-a_1 p_1^{(1)}-\cdots-a_k p_k^{(1)},
\]
where $\frakh$ is the intersection of $C_0$ with a line in the plane. 
Since $\calO_{C_0}(C_0) \cong \calO_{C_0}(3\frakh)$, Lemma \ref{wk} gives  
\[
\calO_{F_0}(F_0) \cong \calO_{F_0}(3e_0-e_1-\cdots-e_9) \cong (\pi|_{F_0})^*(\calO_{C_0}(\frakd)).
\]
This implies that $\calO_{F_0}(F_0)$ is of order $m$ in $\Pic(F_0)$ if and only if $\frakd$ is of order $m$ in $\Pic(C_0)$. If we choose the group law $\oplus$ on the set $C_0^\#$ of regular points of $C_0$ with a nonsingular inflection point $o$ as the zero point, then the latter condition is equivalent to 
\beq\label{sum}
a_1p_1^{(1)}\oplus \cdots \oplus a_k p_k^{(1)} = \epsilon_m
\eeq
where $\epsilon_m$ is a point of order $m$ in the  group $(C_0^\#, \oplus)$. 

This provides a way to construct Halphen pencils and the corresponding Halphen surfaces (under our assumptions that $F_0$ is irreducible). Start with an irreducible plane cubic $C_0$, and choose $k$ points $p_1^{(1)},\ldots,p_k^{(1)}$ in $C_0^\#$ satisfying  Equation \eqref{sum} with $a_1+\ldots+a_k = 9$. Then blow up the points $p_1^{(1)},\ldots,p_k^{(1)}$ together with infinitely near points $p_{j}^{(i)}, i = 2,\ldots,a_j,$ to arrive at a rational surface $\pi:X\to \bbP^2$. Then $|-K_X| = |F_0|$, where $F_0$ is the proper transform of $C_0$. Since the $p_i^{(1)}$ satisfy Equation \eqref{sum},  $\calO_{F_0}(F_0) \cong (\pi|_{F_0})^*\calO_{C_0}(\epsilon_m)$ is of order $m$ in $\Pic (F_0)$. Since $F_0$ is irreducible, $|-mK_X| = |mF_0|$ is nef. Consequently, Proposition   \ref{pro:halphen} shows that $X$ is a Halphen surface.

\subsection{Unnodal Halphen surfaces}\label{general}

By Lemma \ref{-ncurves}, Halphen surfaces  contain no $(-n)$-curves
with $n\geq 3$. Recall that a Halphen surface is {\bf{unnodal}} if it has no $(-2)$-curves. Since a $(-2)$-curve $R$ satisfies $R\cdot K_X = 0$, it  must be an irreducible component of a fiber of the genus $1$ fibration $f:X\to \bbP^1$. 
Conversely, all reducible fibers of $f$ contain $(-2)$-curves, because
$f$ is a relatively minimal elliptic fibration.
Thus {\sl{$X$ is unnodal if and only if all members of the pencil $|-mK_X|$ are irreducible}}.

In this case all the curves $E_i$ are $(-1)$-curves; in particular, there are no infinitely near points in the Halphen set.  Also, in this case, the morphism $f:X\to \bbP^1$ is an elliptic fibration because any quasi-elliptic fibration on a rational surface has a reducible fiber whose irreducible components are $(-2)$-curves (this follows easily from \cite{Dolgachev-Cossec:Book}, Proposition 5.1.6, see the proof of Theorem  5.6.3). The fibers of $f$ are irreducible curves of arithmetic genus $1$.  This shows that {\sl{a Halphen surface is unnodal if and only if it arises from a Halphen pencil with irreducible members}}.  

\begin{proposition}\label{gen}Let $X$ be a Halphen surface of index $m$. Then $X$ is unnodal if and only if the following conditions are satisfied.
\begin{itemize}
\item[(i)] There is no infinitely near point in the 
Halphen set  $\calP = \{p_1,\ldots,p_9\}$;
\item[(ii)] the divisor classes 
\begin{eqnarray*}
&{}&-dK_X+e_i-e_j, \ i \ne j,\quad 0\le 2d \le m,\\
&{}&-dK_X\pm(e_0-e_i-e_j-e_k), \ i < j < k, \  0\le 2(3d\pm 1) \le 3m,
\end{eqnarray*}
are not effective.
\end{itemize}
\end{proposition}

\begin{remark}
Since $K_X=-3e_0+\sum e_i$, the two types of divisor classes in condition
(ii) are equal to 
\begin{eqnarray*}
&{}&3de_0-d(e_1+\cdots+e_9)+e_i-e_j, \ i \ne j,\\
&{}&3de_0-d(e_1+\cdots+e_9)\pm(e_0-e_i-e_j-e_k), \ i < j < k.
\end{eqnarray*}
\end{remark}

\begin{example}
When $m=1$, the inequalities $2d\leq m$ and $2(3d\pm 1) \leq 3m$ lead
to $d=0$ and the conditions are respectively redundant with (i), or exclude triples of collinear
points in the set $\{ p_1, \ldots, p_9\}$.

When $m=2$, the inequality $2d\leq m$ reads $d=0$ or $1$, and we have 
to exclude a cubic through $8$ points with a double point at one of them.
The inequality $2(3d\pm 1) \leq 3m$ gives rise to curves of degree $1$ (for $d=0$)
and degree $2$ (for $d=1$): We have to exclude triples of collinear points,
and sets of six points on a conic.
\end{example}

\begin{proof}[Proof of Proposition \ref{gen}] Suppose the conditions are satisfied. Since no class $e_i-e_j$ is effective, the exceptional curves $E_i$ are $(-1)$-curves. Thus the morphism $\pi:X\to \bbP^2$ is the blow-up of $9$ points, none of which is infinitely near another. This implies that $\pi$ does not contract any component of a member of $|-mK_X|$. 
Let $R$ be a $(-2)$-curve on $X$. It must be an irreducible component of a member of $|-mK_X|$, hence $\bar{R} = \pi(R)$ is an irreducible component of a curve of degree $3m$. Taking a complementary component, we may assume that 
\beq\label{assumption}
2\deg(\bar{R}) \le 3m.
\eeq
The divisor class $r = [R]$ belongs to $K_X^\perp$ and satisfies $r^2 = -2$. Since $(r+dK_X)^2 = -2$, we can change $r$ into $r'=r+dK_X$ in such a way that $r' = d'e_0-k_1'e_1-\cdots-k_9'e_9$ with $|d'|  \le 1$. All such vectors $r'$ can be listed: 
\beq
r' = e_i-e_j,\  i\ne j, \quad {\text{ or }} \quad r'=\pm (e_0-e_i-e_j-e_k), \ i < j < k.
\eeq
Thus 
\begin{eqnarray*}
r & = & -dK_X +e_i-e_j, \quad {\text{or}} \\
 r & = &  -dK_X\pm (e_0-e_i-e_j-e_k), \ i < j < k.
 \end{eqnarray*}
Since $\deg \bar{R} = r\cdot e_0 > 0$ and $-K_X\cdot e_0=3$, the curve $\bar{R}$  is of degree $3d$ in the first case, and of degree $3d\pm 1$ in the second case. Thus,  inequality \eqref{assumption} shows that $2d\leq m$ (resp. $2(3d\pm 1)\leq 3m$). From 
(ii), we deduce that $\bar{R}$ and $R$ do not exist, and that $X$ is general.

Conversely, if one of these divisor classes is effective, then $X$ contains a $(-2)$-curve. This proves the proposition.
\end{proof}

\begin{remark} Let $f:X\to \bbP^1$ be a Halphen elliptic surface of index $m > 1$. The generic fiber $X_\eta$ is 
a genus one curve over the field $\bbK(\eta)$ which has no rational point over this field. The Jacobian variety $\Jac(X_\eta)$ of $X_\eta$ (equal to the connected component of the identity of the Picard scheme of $X_\eta$ over $\eta$) is an abelian variety of dimension 1  over $\bbK(\eta)$. Applying the theory of relative minimal models one can construct an elliptic surface $j:J\to \bbP^1$ with  generic fibers $J_\eta$ isomorphic to 
$\Jac(X_\eta)$.  It is called the Jacobian elliptic surface of $X\to \bbP^1$. In our case, this elliptic surface is a Halphen surface of index $1$. The curve $X_\eta$ is a torsor over $J_\eta$. Its class in the group of isomorphism classes of torsors over $J_\eta$ is uniquely determined by a choice of a closed point $y\in \bbP^1$ and an element $\alpha$ of order $m$ in $\Pic(J_y)$. The fiber  $X_y = mF_y$ is the unique multiple fiber of $f$, the curves $F_y$ and $J_y$ can be canonically identified and the isomorphism class of $\calO_{F_y}(F_y)$ coincides with $\alpha$. Since Halphen surfaces of index 1 are parameterized by an open subset of the Grassmannian $G(2,10)$ of pencils of plane cubic curves, their moduli space is an irreducible variety of dimension 8. The construction of the Jacobian surface shows that the moduli space of Halphen surfaces of index $m > 1$ is a fibration over the moduli space of Halphen surfaces of index 1 with one-dimensional fibers. It is expected 
 to be an irreducible variety of dimension 9.

\end{remark}

\begin{remark} It follows from Proposition \ref{gen} that unnodal Halphen sets of given index form a proper Zariski open subset in the set of all Halphen sets of this index. So, one can say that unnodal Halphen sets or the corresponding Halphen surfaces are {\bf general}  in the sense of moduli.
\end{remark}

\subsection{Automorphisms of a Halphen surface}\label{autoha}
We now discuss a result of Coble (see a modern proof in  \cite{Gizatullin:1980}) which describes the automorphism group of an unnodal Halphen surface $X$ and its image in the group $W_X$. 

Let $X$ be a Halphen surface of index $m$. Since the group $\Aut(X)$ preserves the canonical class $K_X$, it preserves
the linear system $\vert -mK_X\vert$ and permutes the fibers of the Halphen fibration $f:X\to \bbP^1$.
As explained in the Introduction, see Section \ref{par:WXWn}, we identify  $\Pic(X)$ with $\bbZ^{1,9}$ and $W_X$ with the Coxeter group $W_9$. 
The lattice $\bbE_9\cong K_X^\perp $ is isomorphic to the root lattice  of affine type $E_8$; the radical of  $\bbE_9$ is generated by the vector $k_9$ and  the lattice $\bbE_8\cong \bbE_9/\bbZ k_9$ is isomorphic to the root lattice of finite type $E_8$.
Consequently, the Weyl group $W_9$ is isomorphic to the affine Weyl group of type $E_8$, and fits in the extension
\beq\label{iota2}
0\to \bbE_8 \overset{\iota}{\longrightarrow} W_9 \to W_8 \to 1,
\eeq
where $\iota: \bbE_8 \to W_9$ is defined 
by the formula
\beq\label{iota}
\iota(w)(v) = v+(v,k_9)w-\bigl((w,v)+\half(v,k_9)(w,w)\bigr)k_9,
\eeq
and $W_8$ is a finite group of order  $2^7\cdot 3^3\cdot 5\cdot 8!$. \footnote{It is an exercise to check that  $\iota(w+w') = \iota(w)\circ \iota(w')$.}

The following theorem shows that the size of $\Aut(X)$ depends on the
existence of reducible fibers for $f$ (see \cite{Gizatullin:1980} for a more precise statement).  We identify the lattice $\bbE_8$  with $K_X^\perp/\bbZ K_X$ and the map $\iota$ with the homomorphism $K_X^\perp/\bbZ K_X \to W_X$ defined by $\iota(D)(A) = A- (A\cdot D)K_X$. 

\begin{theorem}\label{Halphen-CS}  Let $X$ be a Halphen surface of index $m$.  
If $X$ is unnodal, then $\Aut(X)^*$ contains a subgroup $G$  whose image in $W_9$ is equal to $\iota(m\bbE_8) \subset \iota(\bbE_8)$;  in particular, an unnodal Halphen set is Cremona special.
If $X$ is not unnodal then the index of $\Aut(X)^*$ in $W_X$ is infinite, and thus
$X$ is not Cremona special.
\end{theorem}

\begin{proof} Let $f:X\to \bbP^1$ be the elliptic fibration defined by $|-mK_X|$ and $X_\eta$ be its generic fiber. 
Assume $X$ is unnodal. Then $X_\eta$ is a smooth elliptic curve over the field $\bbK(\eta)$ of rational functions on $\bbP^1$ (see the 
beginning of Section \ref{general}). The closure in $X$ of a closed point $\zeta\in X_\eta$ of degree $d(\zeta) = [\bbK(\zeta):\bbK(\eta)]$ is an irreducible curve $\bar{\zeta}$ such that the restriction of $\pi$ to the curve is a finite cover of degree $d(\zeta)$. Conversely, any irreducible curve $C$ of relative degree $d$ (a $d$-multi-section) intersects $X_\eta$ at a closed point of degree $d$. In the case of an unnodal Halphen surface, any exceptional curve $E_i$ is of relative degree $m$ and the degree of any closed point on $X_\eta$ is a multiple of $m$. In particular, $X_\eta$ has a structure of an abelian variety only if $m = 1$. 

Let $J_\eta = \mathbf{Pic}_{X_\eta/\eta}^0$ be the Jacobian variety of $X_\eta$. Now it is an abelian curve which acts on $X_\eta \cong \mathbf{Pic}_{X_\eta/\eta}^1$ by translations. In particular, $J_\eta(\eta)$ acts by $\eta$-automorphisms on $X_\eta$. Explicitly, the action is defined by the formula
\beq\label{action}
t_\fraka(x) = x' \sim   x+\fraka, \quad \fraka\in J_\eta(\eta),
\eeq
where $x$ is a closed point of $X_\eta$ over the algebraic closure of $\bbK(\eta)$. Since $\fraka$ is defined over $\bbK(\eta)$,  this guarantees that the action is defined over $\eta$. 

It follows from the theory of relative minimal models that the action of $J_\eta(\eta)$ on $X_\eta$ extends to a biregular action of $X/\bbP^1$. 
Thus we obtain an injective morphism
\[
J_\eta(\eta)  \hookrightarrow \Aut(X).
\]
By taking the closure of a divisor on $X_\eta$, we have a natural surjective restriction homomorphism 
\[
\tr:K_X^\perp \to J_\eta(\eta).
\]
Its kernel is equal to the subgroup $\Pic_{\fib}(X)$ of $\Pic(X)$ generated by irreducible components of fibers.  If $X$ is unnodal, $\Pic_{\fib}(X) = \bbZ K_X$, hence $J_\eta(\eta)\cong \bbZ^8$. This proves that $\Aut(X)$ contains a subgroup isomorphic to $\bbZ^8$. 

It remains to find out  how this abelian group $J_\eta(\eta)$ acts on $\Pic(X)$. Any effective divisor $D$ on $X$ is a sum of irreducible components of fibers and the rest, which we call the {\bf{horizontal part}}. The restriction of the fibration $f:X\to \bbP^1$ to the horizontal part is a finite cover of degree equal to 
\[
\deg_f(D) = -mD\cdot K_X.
\]
 We say that an effective divisor is {\bf{separable}} if its  horizontal part is a separable cover of $\bbP^1$. If $\cha(\bbK)$ does not divide $m$,  any effective divisor is separable. 

For any divisor class $D$ on $X$ let $\tr(D)$ be its scheme-theoretical intersection  with the generic fiber$ X_\eta$. If $D$ is separable and irreducible, then $\tr(D)$ is a closed point of $X_\eta$ which is equal to  the sum of $\deg_f(D)$  closed points over the algebraic closure of $\bbK(\eta)$. Thus  
$t_\fraka(\tr(D)) = \tr(D)+\deg_f(D)\fraka$. Let $A$ be an element of $K_X^\perp$ such that $\tr(A) = \fraka$. Then  
\[
t_\fraka(D) \sim   D-m(D\cdot K_X)A \mod \Pic_{\fib}(X).
\]
At this point, we assume that $X$ is unnodal, so that 
\[
t_\fraka(D) \sim   D-m(D\cdot K_X)A +\lambda K_X
\]
for some integer $\lambda$. Intersecting both sides with $D$ and $D' = t_\fraka(D)$, and using that 
$D^2 = D'{}^2$, we obtain $\lambda = \frac{m}{2}(D\cdot A+D'\cdot A)$. Intersecting both sides with $A$, we obtain  $D'\cdot A = D\cdot A-m(K\cdot D)A^2$. Combining the two  formulas, we get 
\beq
t_\fraka(D) \sim D-m(D\cdot K_X)A+[m(D\cdot A)-\frac{m^2}{2}(D\cdot K_X)A^2]K_X.
\eeq
(cf. \cite{Gizatullin:1980}, Proposition 9, where the sign in front of $m^2$ must be changed). This should be 
compared with formula \eqref{iota}.
Since $\Pic(X)$ is generated by separable effective divisors (e.g. by the divisor classes $e_0,\ldots,e_9$), this determines the action of $J_\eta(\eta)$ on $\Pic(X)$. Restricting to $K_X^\perp$, we obtain
$$t_\fraka(D) \sim D+m(D\cdot A)K_X.$$
So, the image  of $K_X^\perp/\bbZ K_X$ in $\Aut(X)$ acts on $K_X^\perp$ as the subgroup $\iota(mK_X^\perp)$ of $W_X$. This proves the first assertion of the proposition.

Assume now that $X$ is not unnodal. The set of $(-2)$-curves on $X$ is not empty, 
and coincides with the set of irreducible components of reducible  fibers of $f$. The group $\Aut(X)$ permutes the elements of this finite set. Thus, a finite index 
subgroup of $\Aut(X)$ fixes all the divisor classes of these $(-2)$-curves. The subgroup of all elements of $\bbE_8\subset W_9$ that fix such a class has infinite index in $\bbE_8$; more precisely, the rank of this free abelian group is at most 
$8-k$, where $k$ is the dimension of the subspace which is spanned by classes of $(-2)$-curves.
 This implies that $\Aut(X)$ has infinite index in $W_9$, and therefore that $X$ is not Cremona 
special.
\end{proof}

\begin{remark} Consider the variety  $\widetilde{\calHa}(m)^{\gen} \subset (\bbP^2)^9$ of ordered unnodal Halphen sets $(p_1,\ldots, p_9)$. It follows from the Hilbert-Mumford numerical criterion of stability that the GIT quotient $\widetilde{\calHa}^{\gen}/\!/\SL(3)$ exists and parameterizes the orbits of unnodal Halphen ordered sets of index $m$. The group $W_9$ acts on this space regularly by means of the Coble action (see \cite{Dolgachev-Ortland:Ast}).  Since  two Halphen surfaces are isomorphic if and only if the corresponding Halphen sets are projectively equivalent, we obtain that the subgroup $m\bbE_8$ acts trivially on this variety, hence the quotient group $G = W_9/\bbE_8 \cong (\bbE_8/m\bbE_8)\rtimes W_8$ acts on the orbit space. 

Since all automorphisms of a Halphen surface preserve the elliptic pencil, we have a natural homomorphism $\rho:\Aut(X)\to \Aut(\bbP^1)$ whose image is a finite subgroup 
preserving the set of points corresponding to singular non-multiple fibers. If $m > 1$, it must also fix the multiple fiber, so the group is a cyclic group which  has orbits of cardinality $\le 12$.   The kernel of $\rho$ is a subgroup of $\Aut(X_\eta)$ which is a finite extension of $J_\eta$ by a cyclic group of order dividing 24 (6 if $\cha(\bbK) \ne 2,3$). If $m \le 2$ the group of order 2 is always present. It corresponds to the automorphism $x\mapsto -x$ of the generic fiber if $m = 1$ and the double cover $X_\eta\to \bbP_\eta^1$ given by a 2-section if $m = 2$. For $X$ general enough, we have  $\Aut(X) \cong \bbZ^8\rtimes (\bbZ/2\bbZ)$ if $m \le 2$ and $\Aut(X) \cong \bbZ^8$ otherwise. The image of the generator of $(\bbZ/2\bbZ)$ in $\Aut(X)^*$ is equal to an element of $W_9$ which is mapped to the center of the group $W_8 = W_9/\iota(\bbE_8)$. 

\end{remark}


\section{Coble surfaces}\label{Coble}


The construction of Cremona special point sets with $10$ points is due to Coble (see \cite{CoblePaper}, \cite{Coble:Book}). 
A {\bf{Coble surface}} is a rational smooth surface $X$ such that the linear system $|-K_X|$ is empty, but $|-2K_X|$ is not. The classification of such surfaces can be found in \cite{DolgachevZhang}. In what follows, we only need the special case, where we additionally assume  that $K_X^2 = -1$ and $|-2K_X|$ consists of an irreducible curve $C$.  So, in this paper, a Coble surface is always assumed to be such a surface. A {\bf{Coble set}}  is a point set $\calP$ such that the blow-up of $\calP$ is a Coble surface.

In this section, we study Coble surfaces, and show that unnodal Coble surfaces  are Cremona special. We
provide a  proof, which works in any characteristic. Most arguments and constructions of this section are used in 
Section \ref{par:GC} where we prove the Main Theorem.

\subsection{From Coble   to Halphen surfaces} 
Let $X$ be a Coble surface, and $C$ be an irreducible curve in the linear system $| -2K_X|$.
By definition, the arithmetic genus $p_a(C)$ satisfies
$$p_a(C) = 1+\frac{1}{2}(C^2+C\cdot K_X) = 1+K_X^2 = 0.$$
Thus $C$ is a smooth rational curve with self-intersection $C^2 = 4K_X^2 = -4$.

\begin{proposition}\label{CobletoHalphen} Let $X$ be a Coble surface and $\pi_E:X\to Y$ be the blowing down of a $(-1)$-curve $E$.
Then 
\begin{itemize}
\item  $Y$ is a Halphen surface of index 2;
\item  $C$ is the proper transform of the fiber $F$ containing $y_0 = \pi_E(E)$;
\item the fiber $F$ is irreducible, and $y_0$ is its unique singular point.
\end{itemize}
Conversely, the blow-up of a singular point of an irreducible non-multiple fiber of a Halphen surface of index 2
is a Coble surface.
\end{proposition}

\begin{proof} Since $K_X^2=-1$, we get $K_Y^2=0$. Let $y_0$ be the image of $E$  by $\pi_E$.
Let $F$ be the image of $C$; since $C$ is irreducible, so is $F$. 
Since $C\cdot E=-2K_X\cdot E=2$, the curve $F$ is singular at $y_0$.
>From $\pi_E^*(F)=C+2E$, we deduce that $F$ is an irreducible curve in the linear system $|-2K_Y|$; in particular, 
$F^2=0$ and $-K_Y$ is nef.

By Riemann-Roch, $h^0(-K_Y) > 0$. If $h^0(-K_Y) > 1$, we can find a curve $G$ from $|-K_Y|$ passing through $y_0$, and thus 
$\pi^*(G)-E$ is effective and  $|-K_X|\ne \emptyset$, contradicting the definition of a Coble surface. Thus $h^0(-K_Y) = 1$ and the unique effective divisor $F_0$ in $|-K_Y|$ does not pass through $y_0$. This implies that $|-2K_Y|$ contains two linearly independent divisors, namely  $2F_0$ and $F$, hence $h^0(-2K_Y) \ge 2$. Since $K_Y^2=0$, $2F_0$ and $F$ are disjoint, and the linear system $|-2K_Y|$ has no fixed components; since
it contains the reduced and irreducible curve $F$, this pencil is irreducible. Thus $Y$ is a Halphen surface of index $2$.

Conversely, let $X$ be obtained from a Halphen surface $Y$ of index 2 as indicated in the assertion of the proposition. Since the irreducible fiber $F$ belongs to $|-2K_Y|$ and  its singular point is blown up, the linear system $|-2K_X|$ is not empty   and contains the proper transform of $F$, which is a smooth rational curve. Moreover,  $|-K_X|$ is empty because $|-K_Y|$ consists of the unique multiple fiber $F_0$ of $Y$ and $F_0$ does not pass through the point which we blow up. So, $X$ is a Coble surface.
\end{proof}

Let $\pi:X\to \bbP^2$ be the composition of $\pi_E:X\to Y$ and the blow-down morphism $\pi':Y\to \bbP^2$ which is described in Proposition \ref{pro:halphen}. The image $F$ of $C$ in $Y$ belongs to $|-2K_Y| = |6e_0-2(e_1+\cdots+e_9)|$, hence its image in the plane is an irreducible plane curve of degree $6$ with $10$ singular points $p_1,\ldots, p_{10}$,  maybe with some infinitely near points. This  set $\{p_1,\ldots,p_{10}\}$ is a Coble set of $10$ points.  It contains a Halphen set of index $2$; the remaining tenth point corresponds to the singular point of $F$. Conversely, starting from a set of $10$ singular points $p_1,\ldots, p_{10}$ of an irreducible curve of degree  $6$, we choose a point $p_i$ such that  no other point is infinitely near it.  The remaining set of $9$ points is a Halphen point set 
of index 2.

\subsection{Unnodal Coble surfaces}

By Lemma \ref{-ncurves}, a Coble surface $X$  has no $(-n)$-curves with $n \ge 3$ except the unique curve in $|-2K_X|$ (with self-intersection $-4$).  
If $X$ is unnodal and $\pi_E:X\to Y$ is the blowing down map of a $(-1)$-curve $E$, then $Y$ is an unnodal Halphen surface of index $2$. Otherwise the pre-image, in $X$, of a component of a reducible fiber would define a $(-2)$-curve or a $(-3)$ curve if $\pi_E(E)$ is a singular point of a reducible fiber of the elliptic pencil on $Y$. 
However, the converse is not generally true; one needs to impose more conditions on the point set to ensure that $X$ is unnodal. 

\begin{theorem}\label{discr} A Coble surface is unnodal if and only if it is obtained from a Coble set $\{p_1,\ldots,p_{10}\}$ satisfying the following 496 conditions:
\begin{itemize}
\item[(i)] no points among the ten points are infinitely near;
\item [(ii)] no three points are collinear;
\item [(iii)] no six points lie on a conic;
\item [(iv)] no plane cubic passes through $8$ points with one of them being a singular point of the cubic;
\item[(v)] no plane quartic curve passes through the $10$ points with one of them being a triple point.
\end{itemize}
\end{theorem}

This result is due to A. Coble \cite{CoblePaper}, (10). A modern proof was sketched  in \cite{Cossec-Dolgachev}, Remark 4.7. We supply the full details here. 

\begin{proof}[Proof of Theorem \ref{discr}]   It is known (and is easy to check) that the lattice $\bbE_{10} \cong K_X^\perp$   is a unimodular even lattice of signature $(1,9)$, hence isomorphic to the orthogonal sum $\bbE_8\perp \bbH$, where $\bbH$ is the hyperbolic plane defined by the matrix 
{\tiny $
\begin{pmatrix}0&1\\
1&0\end{pmatrix}.
$}
Explicitly, the $\bbE_8$-part is generated by the vectors $\boldsymbol{\alpha}_i, i = 0,\ldots, 7$.
The sublattice generated by $\bbE_8$ and $\boldsymbol{\alpha}_{8}$ is isomorphic to $\bbE_9$. The radical of this copy of
$\bbE_9$ and the vector $\boldsymbol{\alpha}_{9}$ generate the $\bbH$-part. 
Suppose  conditions (i)-(v) are satisfied. 
Since there are no infinitely near points, $X$ is obtained from $\bbP^2$ by blowing up 
$10$ distinct points~$p_i$. We denote by $E_i$ the corresponding $10$ exceptional divisors,
and denote by $(e_0, e_1, \ldots, e_{10})$ the standard basis of $\Pic(X)$ (see Section \ref{par:WXWn}).

Let 
$$f_i = 3e_0-(e_1+\cdots+e_{10})+e_i, \ i = 1,\ldots,10.$$
We have $f_i\cdot f_j = 1-\delta_{ij}$. Let $Y_i$ be the Halphen surface of index 2 obtained from $X$ by blowing down the curve $E_i$. Then $f_i$ is the divisor class of the proper transform of the half-fiber of the elliptic fibration on $Y_i$. By Proposition \ref{gen}, the first four conditions guarantee that the  Halphen surfaces $Y_i$  are unnodal. The pencil $|2f_i|$ is equal to the   pre-image of the elliptic pencil on $Y_i$; it contains only one reducible fiber, namely $C+2E_i$. 

We have 
\[
f_1+\cdots+f_{10} = 30 e_0-9(e_1+\cdots+e_{10}) = 3\Delta,
\]
with
\[
\Delta = 10e_0-3(e_1+\cdots+e_{10}) = -3K_X+e_0, \quad \Delta^2 = 10, \quad \Delta\cdot K_X = 0.
\]
The linear system $|\Delta|$ is the proper transform of the linear system of plane  curves of degree $10$ passing through the points $p_i$ with multiplicities $\ge 3$. By counting constants, or applying Riemann-Roch, $\dim |\Delta| \ge 5$. Since the divisor classes $f_i$ are represented by irreducible curves, the divisor class $\Delta $ is nef and big. An  irreducible curve $R$ on $X$ with $R\cdot \Delta = 0$ satisfies $R\cdot f_i = 0,$ for all $i = 1,\ldots,10$, hence must coincide with the curve $C$. Thus $|\Delta|$ defines a  morphism 
\[
\phi:X\to \bbP^{\dim |\Delta|}
\] 
that contracts $C$ onto a point $z$. 

Consider the restriction of the linear system $| \Delta |$ to a general member $F_i$ of $|2f_i|$. It is a linear series of degree $6$. Since 
\[
\Delta-F_i \sim \Delta-2f_i = 4e_0-(e_1+\cdots+e_{10})-2e_i,
\]
by condition (v), this divisor class can not be effective. This shows that the restriction of $\phi$ to $F_i$ is given by a complete linear system of
degree $6$. Thus, $\phi_{| F_i}$ is an isomorphism onto a normal elliptic curve of degree $6$ spanning  $\bbP^{\dim |\Delta|}$, 
$\dim |\Delta| = 5$ and $\phi$ is a birational isomorphism onto a  surface of degree $10$. This surface has a quotient
singularity of type $\frac{1}{4}(1,1)$ at $z=\phi(C)$ .

The remaining arguments follow the proof of Theorem 4.4 from \cite{Cossec-Dolgachev}, almost verbatim. Let $R$ be a $(-2)$-curve and 
$ m_0e_0-m_1e_1-\ldots-m_{10}e_{10}$ be its divisor class. We have $m_0 = R\cdot e_0 \ge 0$ and $m_i = R\cdot E_i \ge 0$. Since
$\Delta = -3K_X+e_0$, we have $\Delta\cdot R = e_0\cdot R = m_0$. Suppose first that $m_0 \le 4$. Then listing all solutions of the diophantine equations
\begin{eqnarray*}
3m_0 & = & m_1+ \ldots + m_{10}, \\
m_0^2 + 2 & = & m_1^2+ \ldots + m_{10}^2
\end{eqnarray*}
with $m_0 \le 4$, we find that $R$ is contained in an exceptional curve coming from an infinitely near point, or is equal to the proper transform of one of the curves corresponding to conditions $(ii)$ to $(v)$, or belongs to the class 
$4e_0-(e_1+\ldots +e_{10})-e_i-e_j-e_k+e_s$ with four distinct indices $i,j,k,s$. All  cases except the last one are prohibited by the assumptions of the theorem. In the last case, the curve $R$ does not intersect  the exceptional curve $E_s$; hence $R$ is coming from a $(-2)$-curve on the Halphen surface $Y_s$ obtained by blowing down $E_s$. It follows from Proposition \ref{gen}, that in this case there exists a conic passing through the six points $p_a$ with $a \ne i,j,k,s$. This is prohibited by condition (iii). Thus, there is no 
$(-2)$-curve $R$ with $m_0\leq 4$.

So, assume now that $m_0 > 4$ for any $(-2)$-curve $R$.  Repeating the argument of the proof of Theorem 4.4 from \cite{Cossec-Dolgachev},  $R\cdot w(\Delta) > 4$ for all $w\in W_X$. Taking $w$ to be the reflection with respect to $R$, we obtain $R\cdot w(\Delta) = w(R)\cdot \Delta = -R\cdot \Delta < 0$, a contradiction. 
\end{proof}

\begin{remark} Theorem \ref{discr} shows that unnodal Coble sets form a Zariski open subset in the set of all Coble sets. In particular, an unnodal Coble surface is a general Coble surface in sense of moduli.
\end{remark}
\subsection{Effective divisors on unnodal Coble and Halphen surfaces}\label{par:roots}
Recall that an element of $\bbE_n$ of norm $-2$ is called a {\bf root}. The lattice is spanned by (simple) roots $\boldsymbol{\alpha}_i$ (cf. Section \ref{par:Minkowski-Lattice}). A {\bf real root} is  a root which belongs to the $W_n$-orbit of one (or any) of these simple roots. It is known that all roots are real if and only if $n \le 10$ (see, for example,  \cite{Dolgachev-Ortland:Ast}, Remark 5, p. 79). 
Let $X$ be a rational surface obtained by blowing up a point set $\calP$ with $\vert \calP \vert=n$. 
Using an isomorphism  $\bbE_n\to K_X^\perp$ defined by a choice of a geometric basis on $X$, we can transfer these definitions to elements of the lattice $K_X^\perp$; since $W_X$ acts transitively on the set of geometric bases, the definitions do not depend on the choice of a geometric basis. 

We say that a root $\alpha$ in $K_X^\perp$  is {\bf effective}, if it can be represented by the divisor class of an effective divisor. Clearly, the divisor class of a $(-2)$-curve is an example of an effective  root.

Later on we shall use the following result   due to M. Nagata \cite{Nagata:II} which applies to Coble and Halphen surfaces. For the sake of completeness, we include a proof.

\begin{lemma}\label{-2} Let $X\to \bbP^2$ be  the blow-up of $n \ge 9$ points such that  $|-mK_X|$  contains an irreducible curve for some $m > 0$. Suppose that $X$ has no $(-2)$-curves. Then,
\begin{itemize}
\item[(0)] for any effective divisor $D$ and any $w\in W_X$, the class $w([D])$ is effective;
\item [(i)]  for any $(-1)$-curve $E$ and any $w\in W_X$, the linear system  $|w(E)|$ consists of  a unique $(-1)$-curve;
\item[(ii)] for any simple root $\alpha_i$ and any $w\in W_X$, the linear system  $|w(\alpha_i)|$  is empty;
\item[(iii)] for any primitive isotropic effective divisor class $f$ and any $w\in W_X$,  the divisor class $w(f)$ is  a primitive isotropic effective divisor class.
\end{itemize}
\end{lemma}

\begin{proof}  Let $C$ be an element of $|-mK_X|$. For any irreducible curve $Z\ne C$, we have $Z\cdot C = -mZ\cdot K_X\ge 0$. By adjunction formula, $Z^2 \ge -2$; since $X$ does not contain $(-2)$-curves, all irreducible curves $Z\neq C$ satisfy $Z^2\geq -1$. Moreover, $Z^2 = -1$ implies that $Z$ is a $(-1)$-curve.

Let $Z$ be a curve with $Z^2 \ge -1$ and let $de_0-d_1e_1-\ldots-d_ne_n$ be its divisor class. We have $Z^2 = d^2-d_1^2-\ldots-d_n^2\geq -1$, and $d = [Z]\cdot e_0 > 0$ unless $Z$ is one of the exceptional curves $E_i$. 
We claim that, for any $w\in W_X$, 
\beq\label{claim}
w([Z])\cdot e_0 > 0, \quad {\text{ unless }}\;  w([Z]) = e_i \; {\text{ for some }}\;  i > 0.
\eeq
We use induction on the length $l(w)$ of $w$ as a word in simple reflections. Write $w = s_{i}w'$ with $l(w') < l(w)$, where $s_i$ 
is the reflection given by the simple root $\alpha_i$ (see \S \ref{par:WXWn}). 
By induction $w'([Z])\cdot e_0 > 0$ or $w'([C]) = e_j$ for some $j > 0$.  In the latter case, $s_{i}(e_j)$ is either equal to  $e_k$ for 
some index $k$ or to $e_0-e_a-e_b$ for some indices $a\neq b>0$; so the assertion  is true in the latter case. In the first case, the claim is 
obvious for $\alpha_i\neq \alpha_0$; so we  assume $\alpha_i = \alpha_0$. Let  $w'([Z]) = de_0-d_1e_1-\ldots-d_ne_n$ with $d\geq 1$. Then $w([Z]) = w'([Z])+(w'([Z])
\cdot \alpha_0)\alpha_0$ implies that $w([Z])\cdot e_0 = 2d-d_1-d_2-d_3.$ 
Assume it is negative, i.e. 
$1\leq d < (d_1+d_2+d_3)/2$. Then 
$$-1\le d^2-d_1^2-\cdots-d_n^2 < \frac{1}{4}(d_1+d_2+d_3)^2-d_1^2-\cdots-d_n^2 $$
$$< \frac{3}{4}(d_1^2+d_2^2+d_3^2)-d_1^2-\ldots-d_n^2.$$ 
This gives
$$-4 < -d_1^2-d_2^2-d_3^2-4(d_4^2-\ldots-d_n^2),$$
hence $d_4 = \ldots=d_n = 0$ and $d_1,d_2,d_3 \le 1$ ; since $2\leq 2 d < d_1+d_2+ d_3$ we obtain 
\[
w'([Z]) = e_0-e_1-e_2-e_3=\alpha_0.
\]
This contradicts $Z^2\geq -1$ and proves the claim. 

(0) Let $[D]$ be the class of an irreducible curve, with $[D]\not \in \bbQ [C]$. For all $w\in W_X$, we get 
$w([D])\cdot K_X  =  [D]\cdot w^{-1}(K_X) = D\cdot K_X = -\frac{1}{m}D\cdot C\le 0$; moreover, $D^2\geq -1$.
By Riemann-Roch, $w([D])$ or $[K_X-w(D)]$ is effective. Intersecting $[K_X-w(D)]$ with $e_0$, and using \eqref{claim}, we see that $[K_X-w(D)]$ is not effective; hence $w([D])$ is effective.  

Let now $D$ be any effective divisor. Write $D$ as a sum of irreducible components 
\[
D=sC+\sum_{j=1}^{k} b_j Z_j
\]
where $s\geq 0$, $b_j>0$, and $Z_j$ is an irreducible curve with $[Z_j]\not \in \bbQ [C]$ for all $1\leq j\leq k$.
Then all classes $w([Z_j])$ are effective. Since $w([D])=s[C]+\sum_j b_j w([Z_j])$, the class $w([D])$ is effective.

(i) Let $E$ be a $(-1)$-curve which we identify with its divisor class. As we saw above, $w(E)$ is effective.  Write an effective representative of $w(E)$ as a sum of irreducible components. Since $w(E)^2 = -1 < 0$, some of the components must have negative self-intersection. Thus we can write 
\beq\label{eff}
w(E) \sim sC+a_1G_1+\ldots+a_k G_k+Z,
\eeq
where $G_1,\ldots,G_k$ are $(-1)$-curves, $Z$ is a sum of irreducible curves $Z_j$ with $Z_j^2 \ge 0$, and at least one of the coefficients $s, a_1, \ldots, a_k$ is positive. Applying $w^{-1}$ we get 
$$
E \sim sC+a_1w^{-1}(G_1)+\ldots+a_kw^{-1}(G_k)+w^{-1}(Z).
$$
>From Property (0), we know that the classes $w^{-1}(G_i)$, $1\leq i\leq k$, and $w^{-1}(Z)$ are effective.
Since $E$ is a $(-1)$-curve,  we have $|E| = \{E\}$. Also $E \ne C$ since $E^2 = E\cdot K_X = -1$ but 
$C^2 = m^2K_X^2$ and $C\cdot K_X = -mK_X^2$ cannot be both equal to $-1$. This implies $s = 0, Z = 0, k = 1$ and $E = w^{-1}(G_1)$; thus $|w(E)|=\{G_1\}$ and Property (i) is proved.

(ii) A simple root $\alpha_i = e_i-e_{i+1}$ is effective if and only if there are infinitely near points, which is excluded since we know  that $X$ has no $(-n)$-curves with $n \ge 2$. The simple root $\alpha_0 = e_0-e_1-e_2-e_3$ is effective if and only if the points $p_1,$ $p_2$, $p_3$ are collinear. The proper transform of the corresponding line is a $(-2)$-curve, hence this root is not effective. Suppose $w(\alpha_i)$ is effective for some index $i\geq 0$ and some element $w$ in $W_X$.
Write an effective representative as in Equation \eqref{eff}: 
\[
w(\alpha) \sim sC+a_1G_1+\ldots+a_k G_k+Z.
\]
Applying $w^{-1}$, and using the same argument as in (i), we get a contradiction.

(iii)  By (0) we know that $f'=w(f)$ is effective. This class is isotropic because $w$ preserves the intersection form. It is primitive because otherwise $f=w^{-1}(f')$ is not primitive. 
\end{proof}

\subsection{Automorphisms of unnodal Coble surfaces} \label{par:Unnodal-Coble}

Here we provide a  proof, valid in any characteristic,  of the following theorem which is implicitly contained in \cite{CoblePaper} (expressed in terms  of projective orbits of point sets, see \S \ref{1.5}); several steps of the proof are used in Section \ref{par:GC}.

\begin{theorem}\label{autocoble} Let $X$ be an unnodal Coble surface. Then $\Aut(X)^*$ contains the subgroup of $W_X$ which is isomorphic to
\[
W_{10}(2) := \{w\in W_{10}: w(v)-v\in 2\bbE_{10}\quad  \text{for all $v\in \Pic(X)$}\}
\]
under the natural identification of $W_X$ with $W_{10}$.
\end{theorem}

\begin{remark}
The subgroup $W_{10}(2)$ is obviously normal, and the quotient group is isomorphic to the finite orthogonal group 
$O^+(10,\bbF_2)$ (see \cite{Dolgachev-Cossec:Book}, Theorem 2.9).
We have to explain the meaning of the notation $O^+(10,\bbF_2)$. Up to conjugacy, there
are only two types of non-degenerate quadratic forms over $\bbF_2^{10}$.
Modulo $2$, the intersection form on $K_X^\perp$ is equivalent to the quadratic form 
\[
x_1x_2+ x_3x_4 + \ldots + x_9x_{10}.
\]
In other words, this quadratic form is of {\bf{even type}}; the number of its isotropic vectors (including the null vector) 
is $2^4(2^{5}+1)$. The notation $O^+(10,\bbF_2)$ is meant to distinguish this quadratic form from the form of {\bf odd type}
$ x_1x_2+ x_3x_4 + \ldots + x_7x_{8}+ x_9x_{10}+x_9^2+ x_{10}^2$, which has only $2^4(2^{5}-1)$ isotropic vectors. 
\end{remark}

To prove Theorem \ref{autocoble}, we use the notations from the proof of Theorem \ref{discr}. 
Consider the elliptic fibrations $|2f_1|$ and $|2f_2|$ defined by the first two isotropic vectors among $f_1,\ldots,f_{10}$. 
>From Proposition \ref{CobletoHalphen}, each of these elliptic fibrations comes from a Halphen surface of index $2$, and each $ |f_i|$ 
is reduced to a unique element; we denote these curves by  $F_1\in |f_1|$ and $F_2\in |f_2|$. Let $|D| = |2F_1+2F_2|$. 

\begin{lemma}\label{lem:System-D} The linear system $|D| = |2F_1+2F_2|$ has no base points and defines a morphism $\phi:S\to \bbP^4$.  
\end{lemma}

\begin{remark}\label{rem:System-D}
The proof of this lemma does not use the fact that $-2K_X$ is effective; it depends only on the intersection properties
of $f_1$, $f_2$ and $K_X$, and the fact that $\vert f_i \vert = \{F_i\}$ for some irreducible curve. This will be used in 
Section \ref{par:System-D-II} (see Lemma \ref{lem:D}).
\end{remark}

\begin{proof}  Consider the following three exact sequences 
 \begin{eqnarray}\label{three}
 &&0 \to \calO_X(F_1)\to \calO_X(F_1+F_2)\to \calO_{F_1}(F_1+F_2) \to 0,\\ \notag
 &&0 \to \calO_X(F_1+F_2)\to \calO_X(2F_1+F_2)\to \calO_{F_1}(2F_1+F_2) \to 0,\\ \notag
 &&0 \to \calO_X(F_1+2F_2)\to \calO_X(2F_1+2F_2)\to \calO_{F_1}(2F_1+2F_2) \to 0.\notag
 \end{eqnarray}
  
Since  $\deg   \calO_{F_1}(F_1+F_2) = 1$ and $p_a(F_1) = 1$, we get $h^0(\calO_{F_1}(F_1+F_2)) = 1$; similarly, $h^0(\calO_{F_1}(F_2)) =1$.

Since $h^0(\calO_X(F_1)) = 1,$   Riemman-Roch formula implies $h^1(\calO_X(F_1)) = 0$. Thus  the first exact sequence shows that $h^0(\calO_X(F_1+F_2)) = 2$, and, by Riemann-Roch, $h^1(\calO_X(F_1+F_2)) = 0$. Since $\deg \calO_{F_1}(2F_1+F_2) = 1$, the second exact sequence gives $h^0(\calO_X(2F_1+F_2)) = 3$. By Riemann-Roch, $h^1(\calO_X(2F_1+F_2)) = 0$. Since $\deg \calO_{F_1}(2F_1+2F_2) = 2$, the third exact sequence gives $h^0( \calO_X(2F_1+2F_2)) =  5$, hence $\dim |D| = 4$.

Since $F_1$ and $F_2$ are irreducible and $h^0(\calO_X(2F_1+2F_2))> h^0(\calO_X(F_i+2F_j))$, we obtain that $|D|$ has no fixed components. Let us now assume that $|D|$ has a base point. Such a point  must lie on $F_1$ or $F_2$ and, without loss of generality, we may assume that it lies on $F_1$. The third exact sequence shows that the restriction map  
\[
H^0(\calO_X(2F_1+2F_2))\to H^0(\calO_{F_1}(2F_1+2F_2))
\]
is a  surjective morphism onto a complete linear system of degree $2$ on $F_1$. 
Thus $|D|$ has no base point on $F_1$, and $|D|$ has no base point at all.
\end{proof}

\begin{lemma}\label{lem:System-D2}
The image of $\phi\colon X \to \bbP^4$ is a surface $S$ of degree $4$.
\end{lemma}

\begin{proof}
Since the map is given by the complete linear system $|D|$, its image spans $\bbP^4$ and its degree divides $D^2 = 8$. So, there are only two possibilities: 
$\phi$ is $1$-to-$1$ and its image has degree $8$, and $\phi$ is $2$-to-$1$ and its image has degree $4$; its image can not have degree $\leq 2$ because every quadric surface of $\bbP^4$ is contained in a hyperplane. 

We know that $|F_1+F_2|$ is an irreducible pencil. Let $P$ be its general member.  It is an irreducible  curve of arithmetic genus $2$. Let $\omega_P$ be its canonical sheaf. Since $\calO_{F_i}(C) \cong \calO_{F_i}(-2K_X) \cong \calO_{F_i}$, the adjunction formula gives 
\[
\omega_P^{\otimes 2} \cong \calO_{P}(2F_1+2F_2+2K_X)\cong  \calO_{P}(2F_1+2F_2).
\]
 The exact sequence
\beq\label{es2}
0\to \calO_X(F_1+F_2) \to \calO_X(2F_1+2F_2) \to \calO_{P}(2F_1+2F_2) \to 0,
\eeq
shows that $|D|$ cuts out on $P$ the bi-canonical linear system,  hence the map $\phi_{|P}$ is of degree 2 onto a  plane conic. In particular, the degree of $\phi:X\to \phi(X)$ cannot be equal to 1, hence it is of degree 2 onto a quartic surface $S$ in $\bbP^4$. 
\end{proof}

 Recall that an {\bf{anti-canonical Del Pezzo surface}} $S$ of degree $d\ge 3$ is a surface of degree $d$ in $\bbP^d$ whose minimal resolution is isomorphic to the blow-up $V$  of $9-d$ points in $\bbP^2$ (maybe infinitely near) with $-K_V$ nef and big. Each such surface $S$ is obtained as the image of $V$ by the map given by the linear system $|-K_V|$. An anti-canonical Del Pezzo surface may have singularities (when $-K_V$ is not ample). They are Du Val singularities (or ADE-singularities). 

It is classically known that a surface of degree $4$ in $\bbP^4$ that spans $\bbP^4$ is either an anti-canonical  Del Pezzo, or a cone over an elliptic curve, or  a projection of a surface of degree 4 in $\bbP^5$ (see \cite{Nagata:I,Nagata:II} and \cite{Dolgachev:Topics}). Since $X$ is rational, and $\phi$ is given by a complete linear system there is only one possibility: $S$ is an anti-canonical  Del Pezzo surface.

\begin{lemma}\label{base} The image $S$ of $X$ under the map $\phi:X\to \bbP^4$ is a Del Pezzo surface of degree 4 with four ordinary double points.
\end{lemma}

\begin{proof} Consider the four curves $F_1, F_2, E_1, E_2$, where  $E_i$ is the $(-1)$-curve corresponding to the points $p_i$, $i=1,2$, in the Coble set defining $X$. By definition, $F_i\cdot E_i=0,$ $i=1,2$, and $F_i\cdot E_j = 1, i \ne j$.
Moreover, $2E_i+C\in |2F_i|$, where $C\in |-2K_X|$,  $C\cdot F_i = 0$, and $C\cdot E_j = 2, i \ne j$. 

The restriction of $|D|$ to each of the curves $F_i$ is of degree $2$, and cuts out a complete linear system of degree $2$
on these curves of arithmetic genus $1$; thus $\phi_{\vert F_i}$ is a $2$-to-$1$ cover of $F_i$ onto a line $\ell_i$, $i=1,2$. 

The restriction of $|D|$ to each of the curves $E_i$ is also of degree $2$. Since $F_i\cdot C=0$ and $C^2=-4$, 
the morphism $\phi$ contracts $C$ onto a singular point $q$ of $S$. 
This point is contained in $\phi(E_1)\cap \phi(E_2)$, and is a ramification point for the maps $\phi_{\vert E_i}\colon E_i\to \phi(E_i)$. Thus, the images
of $E_1$ and $E_2$ are two lines $\ell_3$ and $\ell_4$, that intersect at the singular point $q$.

We infer from the exact sequences \eqref{three} that 
\[
3 = h^0(D-F_1) > h^0(D-F_1-F_2) = 2,
\]
hence $\ell_1\ne \ell_2$. Similarly, we prove that 
\begin{itemize}
\item the four lines $\ell_i$ are distinct;
\item the intersections $\ell_1\cap \ell_2, \ell_2\cap \ell_3, \ell_3\cap \ell_4, \ell_4\cap \ell_1$ are non-empty;
\item all other intersections of two of these lines are empty.
\end{itemize}
 
Let $\Pi_1$ be the plane spanned by $\ell_1$ and $\ell_2$ and $\Pi_2$ be the plane spanned by $\ell_3$ and $\ell_4$; the plane
$\Pi_2$ contains the singularity $q= \ell_3\cap \ell_4=\phi(C)$. Since $\Pi_1$ and $\Pi_2$ intersect in at least two points and do not
coincide (for $\ell_1\cap \ell_3=\emptyset$)  they span a hyperplane $H$ in $\bbP^4$. The pre-image of $H\cap S$ in $X$ is the divisor $F_1+F_2+C+E_1+C+E_2\in |2F_1+2F_2|$. 

It is known that a quartic Del Pezzo surface is equal to the base locus of a pencil of quadrics. The quadrangle of lines is equal to the base locus of the restriction of the pencil to $H$. It is easy to see that this pencil must be spanned by two  quadrics of rank 2, i.e. the union of two planes. It follows from this  that the pencil of quadrics containing $S$ is spanned by two quadrics of rank 3. This implies that $S$ contains  4 singular points of type $A_1$ (ordinary nodes) or 2 singular points of type $A_1$ and one singular point of type $A_3$ (see \cite{Dolgachev-Cossec:Book}, Lemma 0.4.2). In the second case the surface does not contain a quadrangle of lines. Thus we obtain that $S$ is a 4-nodal quartic Del Pezzo surface. Its four nodes are the vertices of the quadrangle of lines. It is known that a 4-nodal quartic Del Pezzo surface is isomorphic to the anti-canonical model of the blow-up of 5 points $p_1,\ldots,p_5$ such that $p_3$ is infinitely near  $p_2$, $p_5$ is i
 nfinitely near $p_4$ and the  points $p_1,p_2,p_3$ and $p_1,p_4,p_5$ are collinear (see \cite{Dolgachev-Cossec:Book}, Proposition 0.4.3). The quadrangle of lines is formed by the images of the classes $e_1, e_3, e_5,$ and  $e_0-e_2-e_4$. 
\end{proof}

\noindent
\begin{proof}[Proof of Theorem \ref{autocoble}] Let $\sigma:S'\to S$ be the blow-up of the point $q=\phi(C)$. The exceptional curve is a $(-2)$-curve $R$ on $S'$. The morphism $\phi$ factors through a finite map $\phi':X\to S'$ of degree $2$. The pre-image of $R$ on $X$ is the curve $C$.   

Let $\tau$ be the automorphism of $X$ defined by the deck transformation of the finite double  cover $\phi':X\to S'$. Note that 
it is defined even when $\cha(\bbK) = 2$. Since the genus 1 fibration $X$ is not a  quasi-elliptic fibration (the image of its general 
member is a conic, but an elliptic curve cannot be  mapped to a rational curve by an inseparable map), the map $\phi'$ is a separable finite morphism. 

The Picard group $\Pic(S')\otimes \bbQ$ of $\bbQ$-divisor classes on $S'$ is of rank $3$. It is generated by the classes of the curve $R$ and the proper transforms of the lines $\ell_1$ and $\ell_2$. For any divisor class $A$ on $X$, we have, in $\Pic(X)\otimes \bbQ$,
\[
A+\tau(A) =  a_1f_1+a_2f_2+a_3K_X.
\]
Suppose  $A\in (\bbZ f_1+\bbZ f_2+\bbZ K_X)^\perp$. Since $f_1,f_2$ and $K_X$ are $\tau$-invariant, we obtain, by intersecting both sides with $f_1,f_2,$ and $K_X$, that $a_1 = a_2 = a_3 = 0$. Thus $\tau(A) = -A$. The sublattice of $K_X^\perp$ spanned by $f_1$ and $f_2$ is isomorphic to the hyperbolic plane $\bbH$. Its orthogonal complement is isomorphic to the lattice $\bbE_8$. Thus we obtain
\beq\label{act}
\tau^* = \id_\bbH\oplus -\id_{\bbE_8}.
\eeq

Now let $w$ be an element of $W_X$, and $f'_i=w(f_i)$, $i=1,2$. 
Since $X$ has no $(-2)$-curves, we can apply Lemma \ref{-2} to obtain that  $f_i'$  are primitive, isotropic
and effective divisor class; since $w$ preserves both $K_X$ and the intersection form, $f'_i\cdot K_X=0$
and $f'_1\cdot f'_2=1$.
Similarly, each class $w([E_i])$, $i=1,2$, is represented by a unique $(-1)$-curve $E'_i$. The curve $E'_i$ does not
intersect $f'_i$ and if one contracts $E'_i$ and apply Proposition \ref{CobletoHalphen}, one sees that $f'_i$
is the pull back of a Halphen pencil of index $2$. 
As above we deduce that  each linear system $|2f_i'|$ is an elliptic pencil, and repeating the argument by taking $(f_1',f_2')$ instead of $(f_1,f_2)$, we obtain an automorphism $\tau_\omega = w\circ \tau \circ w^{-1}$. This shows that $\Aut(X)^*$ contains the minimal normal subgroup containing the involution $ \id_\bbH\oplus -\id_{\bbE_8}$. This finishes the proof of Theorem \ref{autocoble} because this
normal subgroup is isomorphic to $W_{10}(2)$ (this non-trivial result is due to A. Coble, a modern proof can be found in  \cite{Dolgachev-Cossec:Book},  Theorem 2.10.1). 
\end{proof}

\begin{remark} One can show that, for a general Coble surface in moduli sense, the group of automorphisms is isomorphic to $W_{10}(2)$.
\end{remark}

\begin{remark} Assume $\cha(\bbK) \ne 2$. 
Let $p:X'\to X$ be the double cover of $X$ branched along the curve $C\in |-2K_X|$. 
This is a K3-surface which admits an involution whose fixed-point locus consists of  a $(-2)$-curve. Conversely, the quotient of a K3-surface by such an involution is a Coble surface.
The Picard lattice of $X'$ contains the pre-image of the Picard lattice of $X$, it is a primitive sublattice isomorphic to the lattice $\la -2\ra \perp \bbE_{10}(2)$, where $\la -2\ra$ is given by the matrix $(-2)$ and  $\bbE_{10}(2)$ is obtained from the lattice $\bbE_{10}$ by multiplying the values of the quadratic form by $2$. 

Over $\bbC$, it follows from the theory of periods of complex K3 surfaces that the coarse moduli space of such K3 
surfaces exists and is an irreducible variety of dimension $9$. 
In fact, it is one of the two codimension 1 irreducible components 
in the boundary of a compactification of the moduli space of Enriques surfaces. 
A Coble surface can be obtained as a degeneration of an Enriques surface when its K3-cover acquires an ordinary double point.  The theory of periods also provides another  proof of Theorem \ref{autocoble} (see \cite{Nikulin}, Theorem 10.1.2), and shows that the automorphism group of an unnodal Enriques surface contains a subgroup isomorphic to $W_{10}(2)$ and a  general Enriques surface has its automorphism group isomorphic to $W_{10}(2)$. 

Our proof extends to the case of unnodal Enriques surfaces $X$ over a field $\bbK$ of arbitrary characteristic. In this case any embedding of 
$\bbH$ in $\Num(X) \cong \bbH\perp \bbE_8$ defines a separable map of degree 2 on a 4-nodal quartic surface 
(it is separable for the same reason as above: the surface has no quasi-elliptic fibrations). The deck 
transformation acts by formula \eqref{act}, and we finish as in the case of Coble surfaces by invoking Theorem 2.10.1 from \cite{Dolgachev-Cossec:Book}.
\end{remark}


\section{Gizatullin's Theorem and Cremona special point sets of nine points}


Let $X$ denote, as above, a rational surface over $\bbK$;
we denote by $n+1$ its Picard number. In this section, we prove that
Cremona special point sets with $n=9$ points are Halphen sets.

\subsection{Parabolic automorphisms}
Let $g$ be an automorphism of $X$ and $g^*$ the linear transformation of $\Pic(X)$ 
induced by $g$. Since the intersection form on $\Pic(X)$ has signature $(1,n)$ and
the nef cone is $\Aut(X)$-invariant,  there are three possibilities for the isometry $g^*$ (see \cite{Cantat:Acta}, \cite{Gizatullin:1980}). 
\begin{itemize}
\item $g^*$ preserves an ample class $h$, i.e. $g^*h=h$; in this case, a positive iterate
$g^k$ of $g$ acts trivially on $\Pic(X)$ and there is an embedding of $X$ into a projective
space $\bbP^N_\bbK$ such that $g^k$ extends to a projective linear transformation of $\bbP^N_\bbK$.
\item $g^*$ preserves a primitive nef class $h$ with $h^2=0$ but does not preserve any ample class. 
\item $g^*$ does not preserve any nef class and the spectral radius (i.e. the largest possible absolute value of an eigenvalue) of the linear transformation 
$g^*$ on $\Pic(X)\otimes \bbR$ is larger than $1$. 
\end{itemize}

In the first case, one says that $g$ (or $g^*$) is {\bf{elliptic}}, in the second that $g$ is {\bf{parabolic}}, and 
in the third that $g$ is {\bf{hyperbolic}} (or loxodromic). 

When $g$ is parabolic, the class $h$ is, up to a scalar multiple, the unique $g^*$-invariant class in the isotropic
cone of the intersection form. In particular, the assumption that $h$ is nef 
could be removed from the definition (see \cite{Gizatullin:1980}, Proposition 4, or \cite{Cantat:Acta}). 
The following fundamental theorem of M. Gizatullin  describes geometric properties of parabolic automorphisms. 

\begin{theorem}[M. Gizatullin]\label{Giz1} 
Let  $X$ be a rational surface  over an algebraically closed field $\bbK$.
Let $G\neq \{\id_X\}$ be a  group of automorphisms of $X$ such that all elements $g$ in $G\setminus\{\id_X\}$ 
are parabolic. 

Then there exists a unique $G$-invariant primitive nef class $h$ in $\Pic(X)$ with $h^2=h\cdot K_X = 0$ and there exists a unique positive integer $m$ 
such that the linear system $|m h|$ is an irreducible pencil of  curves of arithmetic genus~$1$.  
The relative minimal model of this genus $1$ fibration is a Halphen surface $Y$ of index~$m$. 
\end{theorem}


\begin{remark}
Assume that $\bbK=\bbC$ is the field of complex numbers. Then, Gizatullin's Theorem can be strengthened: A finite index subgroup of $G$ preserves each fiber
of the genus one fibration; it acts as a group of translations along the fibers, with dense orbits in almost all fibers;  if $G$ is cyclic, periodic
orbits are dense in $X$ (see \cite{Cantat-Favre, Cantat:Groups}).
\end{remark}

\begin{remark} Gizatullin's Theorem is stated in \cite{Gizatullin:1980} under the assumption that $\cha(\bbK) \ne 2,3$. As the author points out himself, this assumption is made only to avoid quasi-elliptic fibrations, for which the arguments must be slightly modified but
the same conclusion holds. Moreover,  in our applications,  quasi-elliptic fibrations are not realized.
\end{remark}

\begin{corollary}
Let $X$ and $G$ be as in Theorem \ref{Giz1}. The group $G$ descends to a subgroup of the group $\Aut(Y)$. In particular, it 
 contains a finite index free abelian 
subgroup $G_0$, the rank of which is at most $8$ and  is equal to $8$ if and only if $Y$ is an unnodal Halphen surface.   
\end{corollary}

\begin{proof}
This follows from Theorems \ref{Giz1} and \ref{Halphen-CS}.
\end{proof}


\subsection{Cremona special implies unnodal}

\begin{proposition}\label{P3} A Cremona special rational surface $X$  does not contain $(-2)$-curves, i.e. it is unnodal. 
\end{proposition}

\begin{proof} Denote by $\calP=\{p_1, \ldots, p_n\}$ a point set such that $X$ is the blow-up of 
$\bbP^2$ at $\calP$; this provides a natural morphism $\pi:X\to \bbP^2$, and a geometric basis   $(e_0,\ldots,e_n)$ of $\Pic(X)$.  
Let $\calR_X$ be the set of  effective real roots (see \S \ref{par:roots}). Our first goal
is to prove that $\calR_X$ is  empty.

Let $W_X^{{\sf{nod}}}$ be the subgroup of $W_X$ generated by the reflections $s_\alpha$ with $\alpha\in \calR_X$. Since all reflections 
generating $W_X^{{\sf{nod}}}$ are conjugate under $W_X$ to reflections defined by the simple roots $\alpha_i$,  a  result of Deodhar \cite{Deodhar} shows that 
$W_X^{{\sf{nod}}}$ is a Coxeter subgroup of the Coxeter group $W_X$. Its intersection with $\Aut(X)^*$  is equal to $ \{1\}$ (\cite{Dolgachev:1986}, Proposition 3,\footnote{The assertion is stated for the subgroup generated by the reflections with respect to all $(-2)$-curves, but the proof is extended verbatim to our situation.}). So, if  $\Aut(X)^*$ is of finite index in $W_X$,  the group  $W_X^{{\sf{nod}}}$ must be finite, and $W_X^{{\sf{nod}}}$ is a finite Coxeter group. It follows from the classification of such groups that $\calR_X$ is a finite set and the sublattice $\calN_X$ generated by this set is a negative definite sublattice of $K_X^\perp$; hence $\calN_X\otimes \bbR$   is a
proper  subspace of $K_X^\perp\otimes \bbR$.

The group  $\Aut(X)^*$  leaves the sublattice $\calN_X$  invariant, and acts as a finite group on it. Identifying $K_X^\perp\otimes \bbR$ with $\bbE_n\otimes \bbR$, $\Aut(X)^*$ determines a finite index subgroup of $W_n$ that preserves the proper subspace $\calN_X\otimes \bbR$.

Assume $n\geq 10$. We invoke a theorem of Yves Benoist and Pierre de la Harpe from \cite{Benoist} according to which 
the image of $W_n$ in $\Or(\bbE_n\otimes \bbR)$ is Zariski dense. Since, by assumption,  $\Aut(X)^*$ is a subgroup of finite index in $W_n$, we obtain that the image of $\Aut(X)^*$ is Zariski dense either in $\Or(\bbE_n\otimes \bbR)$ or in its connected component of the identity $\SO(\bbE_n\otimes \bbR)$.\footnote{$\SO(\bbE_n\otimes\bbR)$ is the connected component containing the identity for the Zariski topology, but has two connected components as a real Lie group.} Thus the representation of $\Aut(X^*)$ in $\bbE_n\otimes \bbR$ is irreducible, and hence $\calN_X = \{0\}$.

When $n=9$, the theorem proved by Benoist and de la Harpe asserts that the image of the Coxeter group $W_9$ is Zariski dense in the subgroup $G_9$ of $\Or(\bbE_9\otimes \bbR)$ defined by
$
G_9= \{g\in \Or(\bbE_9\otimes \bbR)\; \vert \; g(k_9)=k_9\}.
$ 
If $N$ is a proper subspace of $\bbE_9\otimes \bbR$ which is fixed by a finite index subgroup of $G_9$, then $N/\bbR k_9$ is trivial.  As a consequence, $\calN_X = \{0\}$ in this case too.

 Now we may assume that $\calR_X$ is empty, in particular, no $\alpha_i$ is an effective divisor class.  Suppose we have a $(-2)$-curve  $R$; it represents an effective non-real root.  Write  $r = [R] = a_0e_0-\sum a_ie_i\in \calR_X$. After permuting the elements of the basis, we may assume that $a_1\ge \ldots \ge a_n$. Intersecting with $e_i$, we obtain that $a_i \ge 0$. Applying Noether's inequality \cite{Dolgachev-Ortland:Ast}, Chapter V, Proposition 5, we get $a_0 < a_1+a_2+a_3$. Since $\alpha_0 = e_0-e_1-e_2-e_3$ is not effective, the three points $p_1,p_2,p_3$ are not collinear. Let $X_{123}$ be the surface obtained by blowing up $\bbP^2$ at
$p_1$, $p_2$, and $p_3$.  
Let $\sigma$ be the standard quadratic involution with fundamental points $p_1,p_2,p_3$. This birational transformation of
the plane lifts to an automorphism of $X_{123}$; thus, the composition of $\pi$ with $\sigma$
is a new birational morphism from $X$ to $\bbP^2$, which provides a new geometric basis for $\Pic(X)$. 
In terms of the first geometric basis $(e_i)$ of $\Pic(X)$, the change of basis corresponds to the linear transformation $s_{\alpha_0}$. As a consequence, if $r$ is the class of a $(-2)$-curve then the class $r'$ defined by
$$r' = s_{\alpha_0}(r) = a_0'e_0-\sum_{i=1}^na_i'e_i,$$
is also the class of a $(-2)$-curve (in the new geometric basis). Moreover 
$$a_0' = 2a_0-a_1-a_2-a_3, a_1' = a_0-a_2-a_3, a_2' = a_0-a_1-a_3, a_3'= a_0-a_1-a_2, $$
and $a_i' = a_i, i > 3.$
Now, by Noether's inequality, we have $a_0' < a_0$. Since $s_{\alpha_0}$ is in $W_X$, we can iterate this process, keeping the assumption that $\calR_X$ 
is empty. We can therefore decrease $a_0$ until it becomes 0. In this case one of the simple roots $\alpha_i$ becomes effective, a contradiction. 
\end{proof}

\subsection{Special point sets of nine points}
\begin{theorem}\label{halthm} Let $\calP$ be a Cremona special point set with $9$ points. Then $\calP$ is an unnodal Halphen set of some index $m$. 
\end{theorem}

\begin{proof}  Let $X$ be the surface obtained by blowing up $\calP$.
Since $\Aut(X)^*$ has finite index in $W_X$, this group contains also a free abelian group of
rank $8$ equal to a subgroup of finite index in $\iota(\bbE_8)$ (see \S \ref{autoha}). This group preserves the isotropic vector $K_X$. By Gizatullin's theorem   there exists a (minimal) positive integer $m$ such that the linear system $\vert -mK_X\vert$ is an irreducible elliptic pencil. Thus $X$ is a Halphen surface of index $m$. By Proposition \ref{P3}, $X$ has no $(-2)$-curves. Thus  $X$ is an unnodal Halphen surface.
\end{proof}

\section{The general case}\label{par:GC}

The main results of this section and the next one are summarized in the following
two theorems. 

\begin{theorem}\label{thm:Coble-if-not-cubic}
Let $\calP$ be a point set of $n\geq 10$  points which is not contained in a cubic curve. 
If $\calP$ is Cremona special, 
then $\calP$ is an unnodal Coble set. In particular, $n$ is equal to $10$.
\end{theorem}

By definition, a \emph{Harbourne set} is a set $\calP$ of $n\ge 9$ nonsingular points on a cuspidal cubic $C_0$ 
over a field of positive characteristic with the following property: If $X$ denotes the surface obtained by blowing up $\calP$
and $C$ denotes the strict transform of $C_0$ in $X$, then $X$ has no $(-2)$-curves and the restriction morphism $\frakr$ from $K_X^\perp$
to $\Pic^0(C)$ has finite image. These point set are Cremona special; this is proved by Harbourne in \cite{Harbourne2}.

\begin{theorem}\label{thm:Harbourne-if-cubic}
Let $\calP$ be a point set of $n\geq 10$ points contained in a cubic curve $C_0\subset \bbP^2$. 
If $\calP$ is Cremona special, then $\bbK$ has positive characteristic, $C_0$ is a cuspidal cubic, 
and  $\calP$ is a Harbourne set.
\end{theorem}

These results imply the Main Theorem.  
In this section, we prove Theorem~\ref{thm:Coble-if-not-cubic}. 

\subsection{A surface $Y$ with ten Halphen pencils}\label{5.1}
Let  $\calP=\{p_1,\ldots,p_{n}\}$ be a Cremona special set, with the convention that $j>i$ if $p_j$ is infinitely near $p_i$. 
We blow up successively each point $p_i$  and denote by 
\[
X= X_n\to X_{n-1}\to \ldots \to X_1 \to \bbP^2
\] 
the sequence of blow-ups.  Let $q_i:X\to X_i$ be the corresponding birational morphism and $q_i^*:\Pic(X_i)\to \Pic(X)$ be the canonical homomorphism of the Picard groups. Obviously $q_i^*(K_{X_i}^\perp) \subset K_X^\perp$. Moreover the image under $q_i^*$ of  the geometric basis of $\Pic(X_i)$ (defined by the morphism $X_i\to \bbP^2$) coincides with the first $i+1$ vectors of the geometric basis of $X$ (defined by the morphism $X\to \bbP^2$). This allows one to identify $W_{X_i}$  with a subgroup of $W_X$. Since $\Aut(X)^*$ is of finite index in $W_X$, its intersection with $W_{X_i}$ is a subgroup of finite index; this subgroup preserves the exceptional curves obtained from the blowing ups of $p_j$ 
for $j> i$, and descends as a group of automorphisms of $X_i$. Thus all surfaces $X_i, i\ge 9,$ are  Cremona special, and all sets $\{p_1, \ldots, p_i\}$, $i\geq 9$, are Cremona special sets. By Theorem \ref{halthm},  the surface $X_9$ is an unnodal Halphen surface. In particular, the first $9$ points do not contain infinitely near points. Since the set of the first $10$ points is Cremona special,   Proposition  \ref{P3} shows that the surface $X_{10}$ has no $(-2)$-curves; hence $p_{10}$ is not infinitely near $p_i$ for $i\leq 9$ and the set $\{p_1,\ldots, p_{10}\}$ contains no infinitely near point. 

Let $i$ be an index between $1$ and $10$ such that $p_{11}$ is not infinitely near $p_i$. 
Consider the sequence of points $p_1$, ..., $p_{i-1}$, $p_{i+1}$, ..., $p_{11}$, and apply the same
argument.  We obtain that $p_{11}$ is not infinitely near the points $p_j$ for $1\leq j\leq 10$. 
By induction on $n$, and permutation of the points, this proves the following lemma. 

\begin{lemma}\label{lem:No-Inf-Near}
If $\calP$ is a Cremona special point set of $n\geq 9$ points, then $\calP$ 
is a proper subset of $n$ distinct points in $\bbP^2$, and all subsets of $\calP$
of $m\geq 9$ points are Cremona special.
\end{lemma}

Let $Y = X_{10}$, $E_i\subset Y$ be  the $(-1)$-curve obtained by blowing up $p_i$, $1\leq i\leq 10$,  and $\sigma_{E_i}:Y\to Y_i$ be the blow-down of $E_i$. Since $Y_i$ is Cremona special, it is an unnodal Halphen surface of some index $m_i$. The pre-image of the elliptic fibration of $Y_i$  is an elliptic pencil  $|m_if_i|$ on $Y $, where  
\[
f_i = -K_Y+e_i=\sigma_{E_i}^*(-K_{Y_i}).
\]
The divisor classes $f_1,\ldots,f_{10}$ are exactly the primitive isotropic vectors in $K_Y^\perp$ which we introduced in the proof of Theorem \ref{discr}. 
We denote by ${\tilde{p_i}}\in Y_i$ the point $\sigma_{E_i}(E_i)$; this point is the pre-image
of $p_i$ under the natural projection $Y_i\to \bbP^2$.

\subsection{Almost all indices are different from 1} 

Suppose that  two indices $m_i$ and $m_j$ are equal to $1$, say $m_1 = m_2 = 1$. Let $B_i\subset Y_i$, $i=1,2$, be the fiber of the Halphen fibration which contains ${\tilde{p_i}}$, and let $a_i$ be the multiplicity of this fiber at ${\tilde{p_i}}$. 
Let $A_i\subset Y$ be the strict transform of $B_i$ under the morphism $\sigma_{E_i}$.
Since the fibers of the Halphen surface $Y_i$ are irreducible, both $A_i$ and $B_i$ are irrreducible.
Then $-K_Y$ is effective and represented by $A_i+(a_i-1)E_i$. Since $K_Y^2 = -1$, the  two curves $A_1+(a_1-1)E_1$ and $A_2+(a_2-1)E_2$ have a common component, and thus $A_1 = A_2 $; as a consequence,  $a_1=a_2=1$, i.e. ${\tilde{p_i}}$ is a smooth point of $B_i$, for $i=1,2$. Denote by $A$ the curve $A_1=A_2$. 
By Lemma \ref{wk}, the normal bundle of $A_i$ is equal to $\calO_{A_i}(-E_i\cap A_i), i = 1,2.$ Thus the points $E_1\cap A$ and $E_2\cap A$ are linearly equivalent on $A$. Since $p_1\neq p_2$ and $A$ is an irreducible curve of arithmetic genus $1$, we get a contradiction. This proves the following lemma. 

\begin{lemma}\label{lem>1}
At least nine of the indices $m_i$, $1\leq i\leq 10$, are larger than $1$.
\end{lemma}

\subsection{An alternative}\label{par:Alternative}
Effective curves with divisor class $f_i$ are total transforms of fibers of the Halphen fibration on $Y_i$ with multiplicity $m_i$. Thus, if $m_i=1$,  the class $f_i$ is represented by an irreducible curve (the total transform of any fiber that does not
contain $p_i$), and if $m_i\geq 2$ the class $f_i$ is represented by an irreducible
curve if and only if ${\tilde{p_i}}$ is not contained in the multiple fiber of $Y_i$. 

\begin{lemma}\label{lem:cubic-curve-alternative}
Either the set $\{p_1, \ldots, p_{10}\}$ is contained in (the smooth part of) a cubic curve $C$ 
and then all $f_i$, $1\leq i\leq 10$, are represented by $C'+E_i$ where $C'$ is the strict transform of 
$C$; or $\{p_1, \ldots, p_{10}\}$  is not contained in a cubic curve, and then 
all classes $f_i$ are represented by at least one irreducible curve.
\end{lemma}

\begin{proof}
If the set $\calP_{10}=\{p_1, \ldots, p_{10}\}$ is contained in a cubic curve $C$, no $p_i$ is a singular
point of $C$, because otherwise the Halphen surface $Y_j$, $j\neq i$, would not be special. 
If the set $\calP_{10}$ is not contained in a cubic curve, then $p_i$ is not contained in the multiple 
fiber of $Y_i$, so that $f_i$ is represented by an irreducible curve. 
\end{proof}

For the remainder of this section, we make the following assumption, where Lemma \ref{lem:cubic-curve-alternative} 
is used to prove the equivalence. 

\begin{itemize}
\item[(A)] \textit{Each primitive isotropic class $f_i$ is represented by an irreducible curve. Equivalently, there is no irreducible cubic curve containing  the set $\{p_1, \ldots, p_{10}\}$.  In particular, the multiple fiber of $Y_i$ does not contain~$\tilde{p_i}$.}
\end{itemize}

In other words, we now assume that $Y$ satisfies the assumption of Theorem \ref{thm:Coble-if-not-cubic}.
Note that  the irreducibility of the cubic curves in Assumption (A) follows from Proposition \ref{P3}. 

\subsection{All indices are different from 1}
Suppose now that one index, say $m_1$, is equal to $1$; by Lemma \ref{lem>1}, $m_i\geq 2$ for $i\geq 2$. As above, let $B_1\subset Y_1$ be the fiber of the Halphen fibration which 
contains the point ${\tilde{p_1}}$, $a_1\leq 2$ be its multiplicity at ${\tilde{p_1}}$, and $A_1$ 
be its strict transform in $Y$; then 
\[
\sigma_{E_1}^*(B_1)=f_1= A_1 + a_1 E_1.
\] 
Let $B_2\subset Y_2$ be the fiber of the Halphen fibration which contains the point ${\tilde{p_2}}$. By our assumption (A), the curve 
$B_2$ is not the multiple fiber of the Halphen surface~$Y_2$.
Let $A_2$ be the strict transform of $B_2$ and $a_2$ be the multiplicity of $B_2$
at ${\tilde{p_2}}$. 
Then 
\[
A_2\sim m_2 f_2-a_2E_2 \sim m_2(-K_Y+E_2)-a_2E_2 \sim -m_2 K_Y + (m_2-a_2)E_2
\]
 and 
\begin{eqnarray*}
A_1\cdot A_2 & = & (-K_Y+(1-a_1) E_1)
\cdot (-m_2 K_Y + (m_2-a_2)E_2) \\
& = & -m_2 + (m_2-a_2) + m_2(1-a_1)\\
& = & -a_2 - (a_1-1) m_2;
\end{eqnarray*}
this number is negative because $a_i\geq 1$.
Since $A_1$ and $A_2$ are irreducible curves, we obtain $A = A_1 = A_2$. 
Since $A_1$ is the strict transform of a cubic curve, this contradicts our standing
assumption (A) and proves the following lemma.

\begin{lemma}
Under assumption (A), all indices $m_i$ are larger than $1$.
\end{lemma}

\subsection{The linear system $|D|= |2f_1+2f_2|$}\label{par:System-D-II}

We now show that all indices $m_i$ are indeed equal to $2$; permuting the indices, it is
sufficient to consider $m_1$ and $m_2$: We already know  that $m_1, m_2 > 1$ and we want to show that $m_1=m_2=2$.
For this purpose, 
we study the linear system $\vert 2 f_1 + 2f_2\vert$.

Since $m_i\geq 2$ for all $i$, the linear system $|f_i|$ consists of a unique element. 
Hence, 
\[
|f_1| = \{F_1\}, \quad |f_2| = \{F_2\},
\]
where the divisors $F_1$ and $F_2$ are irreducible curves of arithmetic genus $1$ by assumption (A). Since 
$F_1\cdot F_2 = 1$, the curves $F_1$ and $F_2$ intersect transversally at one point
$$y_0 = F_1\cap F_2.$$
Denote by $\vert D\vert$  the linear system $| 2F_1 + 2F_2|$. 
 
 \begin{lemma}\label{lem:D} Under assumption (A), the following properties are satisfied. 
 \begin{itemize}
\item[(i)] The linear system $|D| = |2F_1+2F_2|$  satisfies 
\[
 D^2= 8, \, \, \, \dim |D| = 4,
 \]
it has no fixed component, and it has no base point. It defines a morphism $\phi:Y\to \bbP^4$.
 
 \item[(ii)] Let $S = \phi(Y)\subset \bbP^4$ be its image. 
 Then, on each curve $F_i$, $i=1,2$, $\phi$ restricts to
 a $2$-to-$1$ cover onto a line $\ell_i\subset S$. The two lines $\ell_1$ and $\ell_2$ are different.
 
 \item[(iii)]
  The linear system $|F_1+F_2|$ is an irreducible pencil with two base points $y_1\in F_1,y_2\in F_2$. The points 
  $s_i = \phi(y_i), i = 0,1,2$, span a plane  in the projective space $\bbP^4$.
   
 \item[(iv)] The degree of $\phi\colon Y\to \bbP^4$ is equal to $2$, and its image is an anti-canonical Del Pezzo surface $S\subset \bbP^4$ of degree $4$; 
 in particular, $S$ is a normal surface. 
 
 \item[(v)] The pre-image of $s_0$ under the map $\phi$ consists of exactly  one point: $\phi^{-1}\{ s_0\}=\{ y_0 \}$.
 \end{itemize}
\end{lemma}

\begin{proof} (i) The argument is the same as  in the proof of Lemma \ref{lem:System-D}, 
where we used only that $f_1$ and $f_2$ were represented by irreducible curves $F_1$ and $F_2$ that do not move 
(see Remark \ref{rem:System-D}).

(ii)  Each curve $F_i$, $i=1,2$, is an irreducible curve of arithmetic genus $1$.
It follows from the third exact sequence in \eqref{three} that the restriction of $|D|$ to the curve $F_i$ is given by a complete linear system of degree $2$. Thus, for $i=1,2$, the curve $\phi(F_i)$ is a line $\ell_i\subset S$ and $\phi$ restricts to a $2$-to-$1$ cover between $F_i$ and $\ell_i$. 
The same exact sequences show that $\ell_1\ne \ell_2$.

(iii) The first exact sequence in 
\eqref{three} shows that $|F_1+F_2|$ is a pencil of curves of arithmetic genus $2$. 
Since $(F_1+F_2)\cdot F_1=1$, its restriction to $F_1$ (resp. to $F_2$) has degree~$1$. 

Let $\epsilon_i$ be the nontrivial $m_i$-torsion divisor class on $F_i$ such that $\calO_{F_i}(\epsilon_i) \cong \calO_{F_i}(F_i)$. 
For $i=1,2$ we have $\calO_{F_i}(F_1+F_2) \cong \calO_{F_i}(y_i)$, where $y_i$ is an element of the linear system $|y_0+\epsilon_i|$ on $F_i$. Clearly, $y_i\ne y_0$ and does not move because the arithmetic genus of $F_i$ is $1$. This shows that 
$|F_1+F_2|$ has no fixed components and its base points are $y_1$ and $y_2$. 

The restriction of $\phi$ to $F_i$ is given by the linear system $|D|$, i.e. by  
\[
|\calO_{F_i}(2F_1+2F_2)| = |\calO_{F_i}(2y_0+2\epsilon_i)|.
\] 
Since $2y_0+2\epsilon_i \sim y_i+y_0+\epsilon_i\not\sim  y_0+y_i$, we obtain that the points $y_0$ and $y_i$ are not mapped to the same point by $\phi$; thus, 
$s_0=\phi(y_0)$ and $s_i=\phi(y_i)$, $i=1,2$, are distinct. 
As a consequence, the plane $\Pi_0$ spanned by the two lines $\ell_1$ and $\ell_2$ is also spanned by the three points $s_0$, $s_1$, and $s_2$.

(iv) Let us now prove that the degree of $\phi$ is equal to 2, i.e. $\deg S = 4$. Then, as in Section \ref{par:Unnodal-Coble} (see after Lemma \ref{lem:System-D2}), the image $S$ of $\phi$ is a del Pezzo surface of degree 4.  

Since the map is given by a complete linear system, and a quadric surface in $\bbP^4$ is contained in a hyperplane, the only other possibility is that the degree is equal to 1,  i.e. $\deg S = 8$; thus, we now assume $\deg(S)=8$ and seek a contradiction. 

Let $\frako$ be a general point of $Y$; in particular, $\frako$ is not a base point 
of $\vert F_1+F_2\vert$ and its image $\frako'=\phi(\frako)$ in $S$ is not contained in the plane $\Pi_0$. 
Let $\pi:S\to S' \subset \bbP^3$ be the projection from  $\frako'$. The 
composition $\pi\circ\phi$ is a rational map given by the linear subsystem $L$ of $ |2F_1+2F_2|$ of divisors passing through the point 
$\frako$. Replacing $Y$ by its blow-up $Y' \to Y$ at the point  $\frako$, we obtain a birational morphism $\phi':Y'\to S'$ onto a surface of 
degree 7. 

A general member $P$ of the pencil $|F_1+F_2|$ is an irreducible curve of arithmetic genus 2. Consider the restriction of the map $\phi'$ to $P$; it is given by a linear series of degree $4$; would this linear series  coincide with $\vert 2K_P\vert$, then $\phi'_{\vert P}\colon P\to \phi'(P)$ would be a $2$ to $1$ cover onto a conic curve (see the proof of Lemma \ref{lem:System-D2}), and $\phi$ would have degree $2$. Thus, $\phi'_{\vert P}$ is given by  a linear series $|K_P+p+q|$  where $p+q\not\in |K_P|$. 
The image $\phi'(P)$ is a plane quartic curve,  the points $p,q$ are mapped to a singular point of $\phi'(P)$, and $\phi'_{\vert P}$ provides an isomorphism from $P\setminus\{p,q\}$ to $\phi'(P)\setminus\{\phi'(p)\}$.

The unique member $P_0$  from $|F_1+F_2|$ which passes through the point $\frako$ is mapped to a plane  curve on $S$ passing through $\frako'$. This curve has degree $4$,   hence it is projected to a triple line $\ell'_0$ on $S'$. Since $y_1,y_2\in P_0$, the triple line $\ell_0'$ contains the 
projections $s_1',s_2'$ of the points $s_1,s_2$. By assertion (iii), the points  $s_0,s_1,s_2$ span a plane $\Pi_0$ in  $\bbP^4$. Since $\Pi_0$ does not contain the image $\frako'$ it is projected onto a plane $\Pi_0'\subset \bbP^3$ and this plane is spanned by the images $s_0'$, $s_1'$, and $s_2'$ of the three points $s_0,$ $ s_1$, and $s_2$. The triple line $\ell_0'=\phi'(P_0)$ is spanned by $s_1'$ and $s_2'$. 

The image of the pencil $|F_1+F_2|$ on $S'$ is cut out by the pencil $\calQ$ of planes containing $\ell'_0$. The plane $\Pi_0'$  is a member of $\calQ$ and cuts out in $S'$ a curve of degree 7: This curve is equal to the union of the three lines $\ell_1'$, $\ell_2'$, and $\ell'_0$, where the first two enter with multiplicity 2 and the last one with multiplicity 3. Any other plane from $\calQ$ cuts out in $S'$ the line $\ell'_0$ taken with multiplicity 3 and a quartic 
curve $P'= \phi'(P)$ for some $P\in |F_1+F_2|$. Let $N$ be the closure of the set of double points of irreducible plane quartics cut out by $\calQ$ in $S'$. It is a double conic $K$ on $S'$ (its pre-image on $Y$ intersects $P$ at two points, hence intersects $2F_1+2F_2$ with multiplicity 4). It passes through a point on $\ell'_0$, the image of a double point of the plane quartic $\phi(P_0)$.

Let us see what else is in the singular locus of $S'$. Suppose  $Z$ is another irreducible curve in  the singular locus of $S'$ which is not contained in any plane from $\calQ$.  Then a general plane $\Pi'$ from $\calQ$ intersects it, and hence the image of a general $P\in |F_1+F_2|$ in $\Pi'$
 acquires an additional double point. This contradiction  shows that $Z$ is contained in some plane $\Pi'\in \calQ$, certainly, different from $\Pi_0'$. We claim 
that it must be a double conic. The unique alternative possibility is that $Z$ is a double line. Let $F$ be its pre-image in $Y$. It is an irreducible component of some divisor $P$ from $|F_1+F_2|$, with $P=F+R$ for some effective divisor $R$. Since the line $Z$ intersects only one  of the points $s'_1$ and $s'_2$, the curve $F$  passes through only one base point 
of $|F_1+F_2|$, hence we may assume that $F\cdot F_1 = 0$. 
Let $\sigma_{E_1}:Y\to Y_1$ be, as in Section \ref{5.1}, the blow-down of the curve $E_1$. We know that $Y_1$ is a Halphen surface with 
elliptic fibration $|m_1\sigma_{E_1}(F_1)|$, the fibers of which are irreducible. We also know that $F_1 = \sigma_{E_1}^*(-K_{Y_1})$, hence 
$ (\sigma_{E_1})_*(F)\cdot K_{Y_1} = 0$. If  $\sigma_{E_1}(F)$ is not a point,  it must be a fiber of the elliptic fibration $|m_1\sigma_{E_1}(F_1)|$, and it can not be the multiple fiber $\sigma_{E_1}(F_1)$ (because $Z\neq \ell_1'$). Since 
$F\cdot F_2 \leq   1$, this contradicts $m_1>1$, and shows that $\sigma_{E_1}(F)$ is a point, which means that $F$ is the curve $E_1$. In particular, 
$F\cdot F_2=1$, $F_1\cdot R=1$, and $F_2\cdot R=0$, because the self-intersection of $P$ is $2$.  The same reasoning, applied to $\sigma_{E_2}$, implies that $R$ is equal to $E_2$, and provides a contradiction because $E_1+E_2$ is not a member of $|F_1+F_2|$.

So, we have now computed  the singular locus of $S'$. It consists of two double lines $\ell_1',\ell_2'$, one triple line $\ell_0'$, the double conic $K$, and some number $c$ of double conics contained in planes from the pencil $\calQ$. 

A general plane section of $S'$ is a plane curve of degree 7. It has $2+2(c+1)$ double points and one triple point. 
Its geometric genus is equal to $15-3-2-2(c+1)$ which is an even number. On the other hand, the geometric genus is equal to the genus of a general curve from $|2F_1+2F_2|$ which is equal to 5. This contradiction proves  assertion (iv).

(v)  Since $\phi:Y\to S$ is a  map of degree 2 onto a 
 normal surface, the existence of the Stein decomposition of $\phi$ implies that the pre-image of any point consists of at most 
 two points or contains a one-dimensional component. 
 
We want to show that the pre-image of $\{s_0\}$ is reduced to $\{ y_0\}$. It contains the union of  the pre-images of $s_0$ under the restriction maps $F_i\to \ell_i$. Since each such map is of degree $2$ and  $\phi$ does not blow down any curve intersecting $F_1+F_2$, we may assume that $y_0$ is a ramification point of  $\phi_{\vert F_1}\colon F_1\to \ell_1$ and the pre-image of $F_2\to \ell_2$ consists of two points $y_0$ and $y_0'$. 

Let $\nu:\bar{S}\to S$ be the blow-up of the point $s_0$. Suppose $s_0$ is a nonsingular point of $S$. Then the exceptional curve $E$ of  $\nu$ is a $(-1)$-curve on $\bar{S}$. The rational map $\nu^{-1}\circ \phi:S\dasharrow \bar{S}$ extends to a regular map $\bar{\phi}:\bar{Y} \to \bar{S}$ with disconnected  exceptional locus over $y_0$ and $y_0'$. However, the restriction of $\bar{\phi}$ over $E$ is ramified at the point of intersection of $E$ with the proper transform of $\ell_1$, and hence $\bar{\phi}^{-1}(E)$ can not be disconnected. This contradiction shows that $s_0$ is a singular point of $S$, hence $\phi$ is not \'etale at $y_0$, hence $s_0$ is a ramification point of $\phi$ and, as such, has a
  unique pre-image.
\end{proof}
 
 Now we can easily deduce from this lemma that $m_1 = m_2 = 2$. We know that the maps $F_i\to \ell_i$ ramify over $s_0$.  Each map is given 
 by the restriction of the linear system $|2F_1+2F_2 | $. It is equal to the linear system $|2y_0+2\epsilon_i|$, where $\calO_{F_i}(\epsilon_i)\cong \calO_{F_i}(F_i)$ and $m_i$ is the order of $\epsilon_i$ in $\Pic^0(F_i)$. Since $2y_0\in |2y_0+2\epsilon_i|$, we obtain $2\epsilon_i \sim 0$, hence $m_i \leq 2$. Since all indices $m_i$ are larger than one, we get $m_i=2$ for all $1\leq i \leq 10$. 
  
\subsection{The surface $Y$ is a Coble surface}

The surface $Y$ is obtained from $Y_i$ by blowing up a point ${\tilde{p_i}}$; this point is contained in a non-multiple member $B_i$ of the Halphen pencil of $Y_i$. Let $a_i$ be the multiplicity of ${\tilde{p_i}}$  on $B_i$: It is equal to $1$ if ${\tilde{p_i}}$ is a nonsingular point of $B_i$ and $2$ otherwise.  Let ${A}_i$ be the strict transform of $B_i$ on $Y$. Since $m_i=2$, we have 
\beq\label{c2}
{A}_i+a_iE_i \sim 2f_i = -2\sigma_{E_i}^*(K_{Y_i}) = -2K_Y+2E_i.
\eeq 
If $a_i = 2$ for some index, then ${A}_i\in |-2K_Y|$, hence $Y$ is a Coble surface, all $A_i$ coincide, and their natural projection  on $\bbP^2$ is a sextic curve with double points at $p_1$, $p_2$, $\ldots$, $p_{10}$. 

Suppose $a_i =  1$ for all $i$. Then
\beq\label{q}
{A}_i\sim -2K_Y+E_i.
\eeq
 Let $F_i\in |f_i| = |-K_Y+E_i|$ be the pre-image of the reduced double fiber of the Halphen pencil on $Y_i$. Since $\sigma_{E_i}$ is an isomorphism over an open neighborhood of this fiber, the normal bundle $\calO_{F_i}(F_i)$ is of order 2 in $\Pic(F_i)$. Applying \eqref{q}, we obtain 
 \[
 \calO_{F_i}({A}_j-E_j) \cong \calO_{F_i}(-2K_Y) \cong \calO_{F_i}.
 \]
For $j \ne i$, $A_j\cdot F_i = E_j\cdot F_i= 1$, hence  ${A}_j$ and $E_j$ intersect $F_i$ transversally at the same point. 
Let $\pi:Y\to \bbP^2$ be the natural projection, i.e. the blow-down of all curves $E_l$, $1\leq l \leq 10$. 
The image of $F_i$ is an irreducible cubic curve $C_i$.  
By blowing down $E_j$,  we see that $\pi(A_j)$ is tangent to $C_i$ at $p_j$. 
Taking $k\ne i,j$, we obtain that $\pi(A_j)$ is tangent at the same point to $C_k$; 
since $A_j$ intersects $E_j$ transversely, $C_k$ and $C_i$ are tangent at $p_j$.  
Fixing $i$ and $k$ and changing $j$, we see that $C_k$ and $C_i$ are tangent 
at the $8$ points $p_j$ with $j\neq k$, so that the two irreducible cubics $C_i$ and $C_k$ coincide. Hence $F_i=F_k$, while $f_i\neq f_k$: This contradiction shows that
$Y$ is a Coble surface. 
 
Thus we have proved Theorem \ref{thm:Coble-if-not-cubic} when $\calP$ contains
exactly $10$ points. 

\subsection{The surface $X$ coincides with the Coble surface $Y$}\label{par:X=Y}
To conclude the proof of Theorem \ref{thm:Coble-if-not-cubic}, it remains to show that 
$n  = 10$, and thus $X = Y$ is a Coble surface, when $\calP$ is a Cremona special set 
 not contained in a cubic curve.
We may assume  that $X = X_{11}$, so that $X$ is obtained by blowing up $11$ points $p_1,\ldots, p_{11}$ in $\bbP^2$; 
by Lemma \ref{lem:No-Inf-Near} none of them is an infinitely near point.   

By assumption, $\calP$ is not contained in a cubic curve. Let $\calP_j$, $1\leq j\leq 11$, be the subset of $\calP$ 
obtained by removing the point $p_j$. Assume that 
three of the sets $\calP_j$, say $\calP_9$, $\calP_{10}$ and $\calP_{11}$  are contained in 
cubic curves, say $C_9$, $C_{10}$ and $C_{11}$. By Section \ref{par:Alternative}, these
three cubics are irreducible. Let $\calP'=\{p_1, \ldots, p_8\}$.  
 Since $\calP'$   lies on an irreducible cubic curve, no four of its points are collinear.  
 It follows that the linear system of cubic curves containing $\calP'$ is of dimension $1$; let $q$ be the ninth base point of this pencil.  The curves $C_9$, $C_{10}$, and $C_{11}$ belong to this pencil, hence any two of them intersect at $q$. Consequently 
$C_9$ and $C_{10}$ contain the set $\calP'$ and the points $q$ and $p_{11}$. It follows that $C_9=C_{10}$ or that $q=p_{11}$.  
If $C_9=C_{10}$, then $\calP$ is contained in this cubic curve,  if $q=p_{11}$, then $C_{11}$ contains $\calP$; in both cases, $\calP$
is contained in a cubic curve, a contradiction.

Thus, at most $2$ of the $\calP_j$ are contained in cubic curves. We can therefore suppose that $\calP_1$, $\calP_2$, 
$\calP_3$ and $\calP_4$ are not contained in cubic curves. The surfaces obtained by blowing up these sets are Cremona special
and, as such, are unnodal Coble surfaces. In particular, each set $\calP_l$, $1\leq l\leq 4$, determines a unique curve
of degree six with nodes along $\calP_l$. 

Let us first assume that the characteristic of the field $\bbK$ differs from $2$.
Consider the Del Pezzo surface $Z$ of degree 2 obtained by blowing up the last seven points $p_5,\ldots,p_{11}$. We identify the elements of the set  
\[
\calQ = \{p_1,p_2,p_3,p_4\} 
\] 
with points on $Z$. By assumption, we have $4$ curves $C_i$ on $Z$ in the linear system $|-2K_Z|$ with double points at $\calQ\setminus \{p_i\}$; each $C_i$ is the proper transform of the
sextic curve with nodes along $\calP_i$, $1\leq i \leq 4$. 
Consider the map $Z\to \bbP^2$ defined by the linear system $|-K_Z|$. Since  $\cha(\bbK)\ne 2$,  its branch curve is a plane quartic curve $B$ (see \cite{LNM777}, Chapter V.6, page 67). Each curve $C_i$ is equal to the pre-image of a conic $K_i$ in the plane. Since $C_i$ is singular at three points $p_j\in \calQ\setminus \{p_i\}$, the conic $K_i$ is tangent to $B$ at the images $q_j$ of the points $p_j$. In particular, the conics $K_i$ and $K_j$ are tangent to $B$ at  two points.  Consider the cubic curves $K_i+L_i$, where $L_i$ is the tangent line to $B$ at $q_i$. They all pass through $q_1$, $q_2$, $q_3$, and $q_4$ with tangent direction $L_i$ at  $q_i$. Thus they generate a pencil of cubic curves with $8$ base points (four are infinitely near  the points $q_i$). The ninth base point must be the intersection point $q$ of the lines $L_i$. But then three tangents of the conic $K_i$ meet at $q$. This can  happen only if $\cha(\bbK) = 2$, so that $n=10$ and $X=Y$ when 
$\cha(\bbK)\neq 2$.

\begin{remark}
The configuration of $4$ conics with each pair tangent at two points is realized in characteristic $2$. Consider the  3-dimensional linear system  of conics $ax^2+by^2+cz^2+dxy = 0$. Each line through the point $(0:0:1)$ is tangent to all conics in the family. Choose four general points  in the plane. For each subset of three of these points find a unique conic in the family which passes through these points. Then each pair of the four conics are tangent at two points. 

Note that the pencil of cubic curves which we used in this proof defines a quasi-elliptic fibration on the blow-up of the base points. It has $8$ reducible fibers of type III in Kodaira's notation (two smooth rational curves tangent to each other at one point). There are no elliptic fibrations on a rational surface with such a combination of reducible fibers.
\end{remark}

It remains to consider the case when $\cha(\bbK) = 2$. The difference here is that the anti-canonical linear system $|-K_Z|$ defines a separable map of degree 2 whose branch curve is a conic and the condition that the pre-image of a conic is singular is not stated in terms of the tangency to the branch locus. So we have to find another argument. 

Since at most $2$ of the $\calP_j$ are contained in cubic curves, we can assume that $\calP_i$ are not contained in cubic curves for $1\leq i\leq 9$; each of these nine point set is special, and is therefore
a Coble set. Let $R_i$ be the proper transform of the sextic curve with double points at $\calP_i,$ $i = 1,\ldots,9$. The curves $R_i$ are pairwise disjoint $(-4)$-curves and the divisor class of each $R_i$ is divisible by $2$. Let $D$ be the divisor $\sum R_i$, where $i$ runs from $1$ to $9$, and let $\calL$ be an invertible sheaf on $X$ such that $\calL^{\otimes 2} \cong \calO_X(D)$. This sheaf defines an  inseparable double cover  $\pi:Z \to X$  with branch divisor $D$. Recall that this means that $Z$ is  locally isomorphic to $\Spec~\calO_X(U)[T]/(T^2+\phi)$, where $U$ is an affine open set,   $\phi$ is a local equation of $D$ in $U$, and $\pi_*(\calO_X) \cong  \calO_X\oplus \calL^{-1}$.  Since $D$ is a reduced divisor, the cover has only finitely many singularities. The set of singularities supports the  scheme of zeros of a section of $\Omega_X^1\otimes \calL^{\otimes 2}$ (see \cite{Dolgachev-Cossec:Book}, Proposition 0.1.2, \footnote{This is analogous to the formula for
 the number of singular points of a holomorphic foliation on a complex  surface $X$ defined by a section of 
 $\Omega_X^1\otimes \calL^{\otimes 2}$}). The length of this $0$-dimensional subscheme is equal to the second Chern class of the rank 2 locally free sheaf $\calE =\Omega_X^1\otimes \calL^{\otimes 2}$. The standard formula from the theory of Chern classes gives
$$c_2(\calE) = c_2(\Omega_X^1)+c_1(\Omega_X^1)\cdot D+D^2.$$
In our situation, we have $c_2(\Omega_X^1) = 14$ (= the l-adic Euler characteristic of $X$) and $c_1(\Omega_X^1) = K_X$ . 
This gives
$$c_2(\calE) = 14+K_X\cdot D+D^2 = 14+2\times 9-4\times 9=-4 < 0,$$
a contradiction.

Thus, Theorem \ref{thm:Coble-if-not-cubic} is proved in all cases, including $\cha(\bbK)=2$.

\section{Cremona special sets of $n\ge 10$ points lying on a cubic curve}\label{section:cubic}

To establish our Main Theorem, it remains to prove Theorem \ref{thm:Harbourne-if-cubic}: Assuming that $\calP$ 
is Cremona special and lies on an irreducible  cubic curve,  we now prove that the set $\calP$ is one of Harbourne's sets in this case. 
Note that, being   Cremona special, $\calP$ does not contain infinitely near point.

\subsection{Torsion sets of points}\label{par:Harbourne}




\begin{lemma}\label{lem:Torsion-if-Special}
If $\calP$ is a Cremona special set of points on a cubic curve $C_0$, then 
\begin{enumerate}
\item $C_0$ is irreducible;
\item  there exists a positive integer $s$ such that the divisor class of $sp_i$ in $\Pic(C_0)$ 
 does not depend on the choice of $p_i$ in $\calP$; 
\item $9sp_i\in |3s\frakh|$ for all $p_i\in \calP$, 
where $\frakh = c_1(\calO_{C_0}(1))$, the divisor class of the intersection of $C_0$ with a line.
\end{enumerate}
\end{lemma}
 
\begin{proof}[Proof of Lemma \ref{lem:Torsion-if-Special}]
We first assume that the number of points in $\calP$ is equal to $10$.
For all indices $i\in \{1, \ldots, 10\}$, denote by $\calP_i$ the set $\calP\setminus \{p_i\}$. 
Since $\calP$ is Cremona special, so is $\calP_i$. By Theorem \ref{halthm}, there exists a positive 
integer $m_i$ such that $\calP_i$ is a Halphen set of index $m_i$. Since $\calP_i$ is contained
in $C_0$, this implies that $C_0$ is irreducible (otherwise the Halphen surface is not unnodal, hence $\calP_i$ is not  Cremona special), and that
\[
m_i(\sum_{j=1, j\neq i}^{10} p_j)\sim  3m_i\frakh,
\]
where $\frakh$ is the divisor obtained by intersecting $C_0$ with a line.
  Let $s$ be the least common multiple of the $m_i$, and let
$\Sigma$ be the sum of the $10$ points $p_i$. We obtain
\[
s(\Sigma-p_i) \sim 3s\frakh. 
\] 
Hence the divisor class of $sp_i$ does not depend on $p_i$. Summing up these equalities, we get
$
9s\Sigma \sim 30s\frakh
$,
and thus
$
9sp_i\sim 3s\frakh.
$

If there are more than ten points in $\calP$, we consider all subsets of ten points within $\calP$. Properties
(1), (2), and (3) hold for these subsets, and therefore also for $\calP$ (for some positive integer $s$).
\end{proof}

The following definition is due to Harbourne \cite{Harbourne2}.

\begin{definition} A rational surface $X$ with Picard number $\rho(X) \ge 11$ and with $|-K_X| = \{D\}$ for some irreducible reduced curve $D$ is  
{\bf K3-like} if the canonical restriction homomorphism 
\[
\frakr \colon K_X^\perp\to \Pic^0(D) 
\]
has finite image.
\end{definition}

We say that a point set $\calP$ which is contained in  an irreducible reduced cubic curve $C_0$ is a {\bf torsion set} if the blow-up surface $X$ is K3-like.
In that case, the strict transform $C$ of $C_0$ coincides with the unique member $D$ of the linear system $|-K_X|$. 

When $\calP\subset C_0$ is a torsion set, we denote by $X$ the surface obtained by blowing up $\calP$, and
 by  $S_\frakr$  the kernel of the restriction homomorphism $\frakr\colon K_X^\perp\to \Pic^0(C)$. The quotient group $K_X^\perp/S_\frakr$ embeds as a 
finite subgroup in $\Pic^0(C)$; we denote by $m > 1$ the smallest integer such that  $K_X^\perp/S_\frakr$ is contained in the $m$-torsion subgroup 
$\Pic^0(C_0)[m]$. 

\begin{corollary} A Cremona special set of $n\ge 10$ points on an irreducible  cubic curve $C_0$ is a torsion set.
\end{corollary}

\begin{proof} The image of the restriction homomorphism $\frakr\colon K_X^\perp\to \Pic^0(C_0)$ is generated by the divisor classes 
$\frakr(e_0-e_1-e_2-e_3)$, and  $\frakr(e_i-e_j)$ for $i>j$. From Lemma \ref{lem:Torsion-if-Special}, we have 
\[
3s\frakr(e_0-e_1-e_2-e_3) = \frakr(3s\frakh-9sp_1) = 0,\  s\frakr(e_i-e_j) = \frakr(sp_i-sp_j) = 0.
\]
This shows that the image of the restriction homomorphism is finitely generated and is contained in the  
$3s$-torsion subgroup $\Pic(C_0)[3s]$ of $\Pic^0(D)$.
As such, $\frakr(K_X^\perp)$ is finite. \end{proof}

\begin{proposition}\label{P6.4} Let  $\calP$ be a torsion set of $n\ge 10$ points on an irreducible cubic curve $C_0$. 
Then $\calP$ is  Cremona special if and only if it is unnodal.
\end{proposition}

\begin{proof} By Proposition \ref{P3}, the condition that  $\calP$ is unnodal  is necessary. By Theorem 3.2 from \cite{Harbourne2}, it is also sufficient.  

Let us sketch Harbourne's argument. Assume that $\calP$ is unnodal. 
Since $\calP$ is a torsion set, the surface $X$ is K3-like: The kernel 
 $S_\frakr $ or $\frakr$ has finite index in $K_X^\perp$. By definition of $m$, $mK_X^\perp\subset S_\frakr$.
Let $w$ be an element of $W_X$. 
Let $E_i\subset X$ be the exceptional divisor obtained by blowing up  $p_i$. The image of the class $e_i$ of $E_i$ by $w$
is represented by a unique $(-1)$-curve $E_i(w)$ (Lemma \ref{-2}, assertion (i)). This curve intersects $C$ in a unique point $p_i(w)$
because $C\in \vert -K_X\vert$. Thus, $w$ transforms the point set $\calP$ into a new point set $\{p_1(w), \ldots, p_n(w)\}$
of the curve $C_0$. This action is the same as the action of $W_n$ on point set, described in Section \ref{1.5} (see also
\cite{Dolgachev-Ortland:Ast}, \S VI.5). 
On the other hand,  the image of $\frakr$ being finite, there exists a finite index subgroup $G$ of $W_X$ such that
$\{p_1(w), \ldots, p_n(w)\}$ is projectively equivalent to $\calP$ for all $w$ in $G$. From Section \ref{1.5} (see \cite{Dolgachev-Ortland:Ast}, \S VI),
 this implies that $G$ is realized as a subgroup of $\Aut(X)^*$; thus, $X$ is special.
\end{proof}


 \begin{example}[Harbourne's examples, see Example 3.4 from \cite{Harbourne2}] Suppose that $\cha(\bbK) = p > 0$ and $C_0$ is an irreducible cuspidal curve. The group $\Pic^0(C_0)$ is isomorphic to the abelian group $(\bbK, +)$. 
 
 Let $C_0^\#$ be the complement of the singular point. 
 Choose a point set $\calP= \{q_1,\ldots,q_n\}$ of $n\ge 10$ points on $C_0^\#$, denote by $X$ the blow-up of $\calP$ and 
consider the restriction homomorphism 
 \[
 \frakr: K_X^\perp \to \Pic^0(C_0) \cong \bbK.
 \]
Its kernel $S_\frakr$ is equal to $pK_X^\perp$ when the elements $3\frakh-q_1-q_2-q_3$, $q_1-q_2$, ..., and $q_{n-1}-q_n$ of $\Pic^0(C_0)$ are linearly independent
over $\bbF_p\subset \bbK$;  general point set with $n\geq 10$ points satisfy this property. 

Suppose, now, that $S_\frakr = pK_X^\perp$.
Since the divisor class of any $(-2)$-curve lies in the kernel of $\frakr$ and, obviously, does not belong to $pK_X^\perp$, 
the point set $\calP$ is unnodal.   By Proposition \ref{P6.4}, $\calP$ is Cremona special.
 \end{example}

>From Proposition \ref{P6.4} we deduce that  Theorem \ref{thm:Harbourne-if-cubic} is a consequence of  the following statement. 

\begin{theorem}\label{T6.10} 
Let $C_0\subset \bbP^2$ be an irreducible cubic curve. Let $\calP\subset C_0$ be a torsion set 
with $\vert \calP \vert =10$, and let $m$ be the smallest integer such that the image of the restriction morphism
$\frakr$ is contained in $\Pic^0(C_0)[m]$. 
If $\calP$ is unnodal, then $C_0$ is a cuspidal cubic and  $\cha(\bbK) $ divides $3m$.
\end{theorem}
 
Note, in particular, that we  assume in the following that the number $n$ of blow-ups is equal to $10$. 
Thus, $\calP$ is now a torsion set of ten distinct points on an irreducible cubic curve $C_0$.
As above, $C$ denotes the strict transform of $C_0$ in the surface $X$. 
The following lemmas provide a way to decide whether $\calP$ is unnodal or not in terms of  
effective roots $\alpha \in K_X^\perp$ (see Section \ref{par:roots}). 

 \begin{lemma}\label{L6.8} Let $\calP$ be a set of $n = 10$ points on an irreducible cubic curve $C_0$. 
 Let $\alpha\in K_X^\perp$ be an effective root. Any effective  representative of  $\alpha$ contains a $(-2)$-curve 
 as one of its irreducible components. 
 \end{lemma}
 
 \begin{proof} Let $D$ be an effective representative of the root $\alpha$. Write $D$ as a positive sum $lC+\sum F_i$, where the $F_i$ are irreducible components different from $C$.  Assume that no component $F_i$ is a $(-2)$-curve. Since $-K_X$ is represented by the irreducible curve $C$, we have $K_X\cdot F_i \le 0$. The adjunction formula implies that any curve $F_i$ with negative self-intersection is a $(-1)$-curve. Since 
  $(\sum F_i)^2 = (D-lC)^2 = -2-l^2 < 0$, some of the components $F_i$ are $(-1)$-curves. Write 
  $\sum F_i = \calE+\calA$, where $\calE$ is the sum of $(-1)$-components of $D$. We have
 \[
 \calE^2 \le \calE^2 + 2 \calE \cdot \calA + \calA^2 = \left(\sum F_i\right)^2 = -2-l^2,
 \]
 \[
  \calE\cdot C = (D-lC-\calA)\cdot C = l- \calA\cdot C \le l.
  \]
 Write $\calE = \sum k_jF_j$, where the curves $F_j$ in this sum are distinct $(-1)$-curves. 
 The second inequality gives $\sum k_j \le  l$ and the first one gives $\sum k_j^2 \ge 2+l^2$. This implies
 $2+l^2 \le \sum k_j^2 \le (\sum k_j)^2 \leq l^2$, a contradiction.
 \end{proof}
  
It follows from this lemma that $X$ is unnodal if and only if it does not contain effective roots.

\begin{lemma}\label{L6.9} Let $\calP$ be a torsion set of $n = 10$ points on an irreducible cubic curve $C_0$.
The surface $X$ contains a $(-2)$-curve if and only if there exists a root $\alpha$ in the kernel $S_\frakr$ of the 
restriction homomorphism $
\frakr:K_X^\perp\to \Pic^0(C)$. \end{lemma}

\begin{proof} Let $\alpha$ be a root. Consider the exact sequence
\[
0\to \calO_X(K_X+\alpha) \to \calO_X(\alpha)\to \calO_C(\alpha) \to 0,
\]
where the first non-trivial map is the multiplication by a section of $\calO_X(-K_X)$ which vanishes along the curve $C$.
Riemann-Roch Formula, applied to the divisor class $-\alpha$, and Serre's Duality imply
\beq\label{rr}
h^0(-\alpha)+h^0(K_X+\alpha) = h^1(K_X+\alpha).
\eeq
Suppose $\alpha$ is the class of a $(-2)$-curve $R$. Since $C$ is an irreducible curve and $C\cdot R = 0$, the curve $R$ is disjoint from $C$. Hence $\frakr(\alpha) = \calO_C(\alpha) \cong  \calO_C$, and thus $\frakr(\alpha)=0$ in $\Pic(C)$.  Conversely, suppose $\calO_C(\alpha) \cong \calO_C$. Then $h^0(\calO_C(\alpha)) \ne 0$, hence either $h^0(\calO_X(\alpha)) \ne 0$, or $h^1(\calO_X(K_X+\alpha)) \ne 0$. In the second case, 
\[0 < h^1(K_X+\alpha)= h^0(K_X+\alpha)+h^0(-\alpha) \le h^0(\alpha)+h^0(-\alpha)\]
implies that $\alpha$ or $-\alpha$ is effective. Thus, in both case, either $\alpha$ or $-\alpha$ is an effective root and Lemma \ref{L6.8} implies that $X$ contains a $(-2)$-curve.
\end{proof}

\subsection{The reduction to a question on the arithmetic of quadratic forms}
The proof of Theorem \ref{T6.10} that we now describe is rather delicate, and uses strong approximation 
results in a specific situation; unfortunately, we have not been able to find a simpler
geometric argument.

We assume that $\calP\subset C_0$ is a torsion set and make use of the notations introduced
in the previous section.  We assume that either (i) $C_0$ is not cuspidal
or (ii) $C_0$ is cuspidal but $\cha(\bbK)$ does not divide $3m$, and our goal is to show that
$X$ contains a $(-2)$-curve; by Lemma \ref{L6.9}, all we need to prove is the existence
of a root $\alpha$ in the kernel $S_\frakr$ of the morphism $\frakr$. 


Note that, under our assumptions $\Pic^0(C)[m] \cong (\bbZ/mZ)^a$, where $a\le 2$. 
Obviously,  $S_\frakr$ contains the sublattice $mK_X^\perp$. Define 
$V_m= S_\frakr/mK_X^\perp$  and $L_m= K_X^\perp/mK_X^\perp$;   $V_m$ is a submodule of $L_m$ over the ring $\bbZ/m\bbZ$:
\[
V_m = S_\frakr/mK_X^\perp \subset L_m = K_X^\perp/mK_X^\perp \cong (\bbZ/m\bbZ)^{10}.
\] 
This submodule $V_m$ contains a free $(\bbZ/m\bbZ)$-submodule of dimension $8$ because $a\le 2$. 
The morphism $\frakr$ induces a morphism from $L_m$ to $\Pic^0(C)[m]$, and 
Theorem \ref{T6.10} becomes a consequence of  the following purely arithmetic statement.
 
 \begin{theorem}\label{T6.10bis} Let $m$ be a positive integer.  Let $L_m = \bbE_{10}/m\bbE_{10}$ and $V_m$ be a $(\bbZ/m\bbZ)$-submodule of $L_m$ containing a free submodule of rank $8$. Then there exists a root   in $\bbE_{10}$ whose projection into $L_m$ is contained in $V_m$.
\end{theorem}

We employ the  theory of quadratic forms over any  commutative ring $A$. Let $M$ be a finitely generated $A$-module. A function $q:M\to A$ is called a {\bf quadratic form} if 
\begin{itemize}
\item for any $x\in M$ and $a\in A$, \ $q(ax) = a^2q(x)$;
\item $b_q(x,y)= q(x+y)-q(x)-q(y)$ is a symmetric bilinear map $M\times M \to A$. 
\end{itemize}
The bilinear form $b_q$ is called the associated symmetric bilinear form. If $2$ is invertible in $A$, then $q = 2^{-1}b(x,x)$, and the notion of a quadratic form is equivalent to the notion of a symmetric bilinear form. A module $M$ equipped with a quadratic form  $q$ is called a {\bf quadratic module}.

We denote by $\bl$ the bilinear form on $K_X^\perp$ given by the intersection product, and we 
 equip $L_m$ with the symmetric bilinear form $\bl_m$ obtained by  reduction of  $\bl$  modulo $m$. When $m$ is even, we equip $L_m$ with the structure of a quadratic module by setting 
\[
q_m(x) = \frac{1}{2}b_m(x,x)
\] 
(recall that our lattice $K_X^\perp$ is an even unimodular lattice). 
Let $\Or(L_m)$ be the orthogonal group of $(L_m,\bl_m)$ (resp. of $(L_m,q_m)$ when $m$ is even).

To prove Theorem \ref{T6.10bis}, fix a root $\alpha$ and consider its image $\bar{\alpha}$  in $L_m$. Suppose we find an element $\sigma\in \Or(L_m)$ such that $\sigma(\bar{\alpha})\in V_m$. Suppose, moreover, that $\sigma$ lifts to an element $w\in \Or(K_X^\perp)$. Then $w(\alpha)$ is a root, $w(\alpha)$ is contained in $S_\frakr$, and we are done. We now develop this strategy. 

\subsection{Orthogonal and Spin groups modulo ${p^l}$} 
Since the quadratic form $\bl$ of the lattice $\bbE_{10}$ is unimodular, there is a connected  smooth group scheme $\bfSO_\bl$ over 
$\bbZ$ such that $ \bfSO_\bl(\bbZ)$ coincides with the group of isometries of the lattice $\bbE_{10}$ with determinant $1$. 

The universal cover of $\bfSO_\bl$ is a smooth group scheme $\bfSpin_\bl$ over $\bbZ$; the group
$\bfSpin_\bl(\bbZ)$ is the  group $\Spin(\bbE_{10})$ of all invertible elements of the even part of the Clifford algebra of the quadratic $\bbZ$-module $\bbE_{10}$ such that the corresponding inner automorphism  leaves invariant $\bbE_{10}$ (see \cite{Knus}, Chapter IV, \S 5).

There is an exact sequence of group schemes
\begin{equation}\label{eq:Spin-to-SO}
1\to \bfmu_2\to \bfSpin_{\bl} \to \bfSO_{\bl} \to 1
\end{equation}
where $\bfmu_2$ is the group of square roots of $1$. 
For any commutative ring $A$ with $\Pic(A) = 0$, the exact sequence defines the following exact sequence of groups
\[
1\to \bfmu_2(A)\to \bfSpin_{\bl}(A) \to \bfSO_\bl(A) \to A^*/A^*{}^2
\]
(see \cite{Knus},  Theorem 6.2.6). This will be applied to  $A = \bbZ/p^k\bbZ$ for prime numbers $p$. In this case $|A^*/A^*{}^2| = |\bfmu_2(A)|$; this number is equal to $2$ if $p \ne  2$, and to $1$, $2$, or $4$ if $p=2$ and $k=1$, $k=2$, or  $k\geq 3$ respectively. Thus the image of $\bfSpin_{\bl}(A)$ in $\bfSO_\bl(A)$ is of index at most $4$.

The group scheme $\bfSpin_\bl$ is requested in order to apply   the Strong Approximation Theorem (see \cite{Kneser}, Theorem 24.6):
 
 \begin{theorem}[Strong Approximation]\label{sat} Let $M$ be a unimodular indefinite integral quadratic lattice of rank $\ge 3$; if $p$ is a prime integer, denote by $M_p = M\otimes \bbZ_p$ its $p$-adic  localization at  $p$. Let $p_i$, $i\in I$, be a finite set of prime numbers. Then the canonical homomorphism  
 \[
 \Spin(M) \to \prod_{i\in I}\Spin(M_{p_i})
 \]
 has a dense image.
 \end{theorem}
 
Note that, by the Chinese Remainder Theorem,
\[
\Or(L_m) = \prod_i \Or(L_{p_i^{k_i}}),
\]
where $m = \prod p_i^{k_i}$ is the prime factorization of $m$. 

Since the fibers of the natural homomorphism $\Spin(M_{p_i})\to \Spin(M/{p_i}^{k_i})$ are open subsets of $\Spin(M_{p_i})$, we obtain a commutative diagram
\beq
\xymatrix{\Spin(\bbE_{10})\ar[r]\ar[d]& \bfSO_\bl(\bbZ) \ar[d]\\
\Spin(\bbE_{10}/m\bbE_{10})\ar[r]&\SO(L_m),}
\eeq
for which the first vertical arrow is surjective, by the Strong Approximation Theorem. Suppose we find a root $\alpha \in \bbE_{10}$ and an element $\sigma$ in the image of the bottom horizontal arrow such that $\sigma(\bar{\alpha})\in V_m$. Then, the commutative diagram shows that  $\sigma$ can be lifted to an element $w\in \bfSO_\bl(\bbZ)$ such that  the projection of $w(\alpha)$ into $L_m$ is contained in $V_m$, and therefore Theorem \ref{T6.10bis} is proved. 

\subsection{Orbits of $\SO_{\bl}$ and $\Spin_{\bl}$ modulo $p$}\label{par:prelim-orbits}
This section is a warm up for the following ones. Fix a root $\alpha$ in $\bbE_{10}$ and  denote by $\bar{\alpha}$ 
its image in $L_p$, where $p$ is  a prime integer.

Since  $L_p$ is a non-degenerate quadratic space over $\bbF_p$, the rank of any submodule of $L_{p}$ on which the quadratic form vanishes identically modulo $p$ is at most $5$. Any  quadratic form of rank $\ge 2$ over $\bbF_p$ represents all elements in $\bbF_p^*$ (\cite{Serre}, Chapter I). Thus, 
$b_p$ represents all elements of $\bbF_p^*$.

Let $r\in V_p$ be an element with $r^2:= b_p(r,r) = -2$.  
If $p\ne 2$, Witt's Theorem provides an element $\sigma\in \Or(L_p)$ such that $\sigma(\bar{\alpha}) = r$.  

If  $p = 2$, we consider  $L_2 = \bbE_{10}/2\bbE_{10} \cong \bbF_2^{10}$ as a quadratic space with the quadratic form  
$q_2(x) = \frac{1}{2} b(x,x) \mod 2$. Since $V_2$ is of dimension $\ge 8$,  $q_2$ does not vanish identically on $V_2$.
There are $496 = 2^4(2^5-1)$ vectors in the set $q_2^{-1}(1)$ and all of them are represented by roots (see \cite{Cossec-Dolgachev}, Remark 4.7).  Thus there are elements of $V_2$ which are represented by roots. 

As a consequence, {\sl{there is a root $\alpha$ in $\bbE_{10}$, such that, for each prime number $p$, there is an element $\sigma\in \Or(L_p)$ satisfying $\sigma(\bar{\alpha}) \in V_p$.}}

We now extend this idea to the case $m = p^k$ with $k > 1$, and use this extension to prove Theorem \ref{T6.10bis}.

\subsection{Quadratic modules and orbits modulo $p^k$}

Here, the ring is  $A = \bbZ/p^k\bbZ$ and the quadratic module is $M = V_m$, equipped with the symmetric bilinear form $b_{p^k}$ or the quadratic form $q_{p^k}$ when $p = 2$. Since $A$ is local, with maximal ideal  $\frakm = pA$, an element $a$ in $A$ is invertible if and only if $a$ is not contained in $\frakm$.

\begin{lemma}\label{lem:quad-values} 
The quadratic module $V_{p^k}$ represents all invertible elements of the ring $\bbZ/p^k\bbZ$ (i.e. $\forall a\in \bbZ/p^k\bbZ$, there exists $x\in V_{p^k}$ such that $q(x) = a$). The same property holds for all free quadratic submodules $M_0\subset V_{p^k}$ of rank~$8$.
\end{lemma}

\begin{proof} 
We prove the lemma for $V_{p^k}$ by induction on $k$. The same proof applies for submodules of $V_{p^k}$ of rank $8$.  

When $k=1$, $V_p$ represents all non-zero elements of the field $\bbF_p$, as explained in Section \ref{par:prelim-orbits}. 
More precisely, for all $a\in \bbF_p^*$ there exists a pair of vectors $(v,w)$ in $V_p$ such that $q(v)=a \mod p$ and $b_q(v,w)=1 \mod p $.

Let $k$ be a positive integer. The induction hypothesis asserts that for all invertible elements $a$ of $\bbZ/p^k\bbZ$
there exists a pair of vectors $(v,w)$ in $V_{p^k}$ such that 
\[
q(v)=a \mod p^k  \quad \text{and} \quad b_q(v,w)=1 \mod p^k.
\]
Let $a$ be an invertible element of $\bbZ/p^{k+1}\bbZ$. Apply the induction hypothesis to find elements $v$ and $w$ of  $V_{p^{k+1}}$ such that $q(v)=a \mod p^k$ and $b_q(v,w)=1 \mod p^k $. Then $b_q(v,w)$ is invertible in $\bbZ/p^{k+1}\bbZ$ and changing $w$ in one of its multiple we construct
a vector $w$ such that $b_q(v,w)=1 \mod  p^{k+1} $. Write $q(v)=a+ bp^k$ and change $v$ into $v'=v-bp^kw$; then $q(v')=a\mod p^{k+1}$. We still have $b_q(v',w)=1\mod p^k $, so that a multiple $w'$ of $w$ satisfies $b_q(v',w')=1$. 
The Lemma is proved by induction.
\end{proof}

Now we invoke the following analog of Witt's Theorem for quadratic modules over local rings (see \cite{Kneser}, (4.4)).

\begin{lemma}[Witt's Theorem] Let $M$ be a quadratic  module over a local ring $A$ with maximal ideal $\frakm$. Let $F,G$ be free primitive submodules of   $M$ over the ring $A$. Any isomorphism of quadratic modules $F\to G$ extends to an automorphism of the quadratic module $M$.
\end{lemma}

Here primitive means that the quotient module is free. 

\subsection{Proof of Theorem \ref{T6.10bis}}

We are ready to prove Theorem \ref{T6.10bis}, hence Theorem \ref{T6.10}, and Theorem \ref{thm:Harbourne-if-cubic}.

Let $m$ be a positive integer and $\prod_i p_i^{k_i}$ be its decomposition into prime factors.
Let $\alpha\in \bbE_{10}$ be a root (for example $\alpha=e_1-e_2$).

Let $p^k$ be any of the factors $p_i^{k_i}$. Consider the image $\bar{\alpha}$ of $\alpha$ in $L_{p^k}$. 
Fix a free submodule $M_0$ of rank $8$ in $V_{p^k}$.
By Lemma \ref{lem:quad-values}, there is an element $v\in M_0$ with $q_{p^k}(v) = -2\in \bbZ/p^k\bbZ$ if $p\ne 2$ and $q_{p^k}(v) = 1$ 
if $p=2$. The element $v$ generates a free primitive submodule of $L_{p^k}$. By Witt's Theorem, we find an element $\sigma\in \Or(L_{p^k})$ such that $\sigma(\bar{\alpha}) = v$; in particular, $\sigma(\bar{\alpha})$ is contained in $V_{p^k}$. 
Let us show that  $\sigma$ can be chosen in the image of the map $\Spin(L_{p^k})\to \Or(L_{p^k})$. 

Recall that the reflection $s_h$ with respect to a vector $h$, for which $q(h)$ is invertible, is defined by the formula 
\[
s_h(x) = x-\frac{b(x,h)}{q(h)}h.
\]
If $h$ is in $M_0$, $s_h$ is an isometry of $L_{p^k}$ that preserves $M_0$. 
By Theorem (4.6) from \cite{Kneser}, any isometry of $M_0$ (resp. $L_{p^k}$) is the product of reflections in elements from $M_0$ (resp. $L_{p^k}$). As explained in \cite{Kneser}, page 39, Section 8, an element $\eta$ of $\Or(L_{p^k})$ is in the image of $\Spin(L_{p^k})$ if and only if $\eta$ is a product of an even number of reflections, $\eta=s_{h_1}\circ \ldots s_{h_{2s}}$, and its spinor   norm ${\text{\sc{sn}}}(\eta)$ is $1$, i.e. 
\[
{\text{\sc{sn}}}(\eta):=q(h_1)\cdots q(h_{2s}) = 1 \mod (\bbZ_{p^k}^*)^{2}.
\]
(see also \cite{Knus}, chapter IV, \S 6) 

Write $\sigma$ as a composition of reflections $s_{h_i}$, $1\leq i \leq l$,  with $q(h_i)\in \bbZ_{p^k}^*$ ($l$ may be odd). Apply Lemma \ref{lem:quad-values} to find vectors $h_{l+1}$ and $h_{l+2}$ in $M_0$ 
such that 
\[
q(h_{l+1}) = \prod_{i=1}^l q(h_i)^{-1} \quad {\text{and}} \quad  q(h_{l+2})=1\mod (\bbZ_{p^k}^*)^{2}.
\]
Change $\sigma$ into $s_{h_{l+1}}\circ \sigma$ or $s_{h_{l+2}}\circ s_{h_{l+1}}\circ \sigma$ to obtain an isometry which is a product of an even number of reflections. After such a modification, ${\text{\sc{sn}}}(\sigma)=1$ and $\sigma$ is in the image of $\Spin(L_{p^k})$. Since $M_0$ is preserved by $s_{h_{l+1}}$ and $s_{h_{l+2}}$, $\sigma(\bar{\alpha})$ is contained in $M_0$, and thus in $V_{p^k}$.

Let now $\bar{\alpha}$ denote the image of $\alpha$ into $L_m$.
Since the previous argument applies to all prime factors $p_i^{k_i}$ of $m$, the Chinese Remainder Theorem shows the existence
of an element $\sigma$ in $\Spin(\bbE_{10}/m\bbE_{10})$ such that $\sigma(\bar{\alpha})\in V_m$.
By the Strong Approximation Theorem, $\sigma$ lifts to an element $\sigma'$ in $\Spin(\bbE_{10})$; then, the image $w$ of
$\sigma'$ in $\SO_{\bl}(\bbZ)$ maps the root $\alpha$ onto a root $w(\alpha)$ whose projection modulo $m$ is in $V_m$. This proves
Theorem \ref{T6.10bis} and hence Theorem \ref{T6.10}.

\section{Non algebraically closed fields and other surfaces}\label{Section:GeneralFields}

\subsection{Non algebraically closed fields} In this Section, the ground field $\bbK$ is not necessarily algebraically closed. Let $\overline{\bbK}$ be its algebraic closure and let $\bar{X} = X\otimes_\bbK \overline{\bbK}$ be obtained from $X$ by  base field change.  Let $W_{\bar{X}}$ denote  the Coxeter
subgroup of $\Or(\Pic(X_{\overline{\bbK}}))$, as defined in Section \ref{par:WXWn}.  We have a sequence of inclusions 
\[
\Aut(X)^*\subset \Aut(\bar{X})^*\subset W_{\bar{X}}\subset \Or(K_{\bar{X}}^\perp).
\]
We say that $X$ is {\bf{Cremona special over $\bbK$}} if $ \Aut(X)^*$
has finite index in $W_{\bar{X}}$ and $W_{\bar{X}}$ is infinite.  The following result extends the Main Theorem to arbitrary fields $\bbK$.

\begin{theorem}
  If $X$ is Cremona special over $\bbK$, then 
\begin{itemize}
\item $\bar{X}$ is  unnodal;
\item $X$ is obtained from $\bbP^2_{\bbK}$ by blowing up a finite subset of $\bbP^2(\bbK)$;
\item $X$ is a Halphen, a Coble, or a Harbourne example (over $\mathbb{K}$).
\end{itemize}
\end{theorem}

\begin{proof}
Let $n+1$ be the Picard number of $\bar{X}$.
Assume that $X$ is Cremona special over $\bbK$; in particular $X$ is obtained from $\bbP^2_\mathbb{K}$ by 
blowing up a $0$-cycle $\calP$ of length $n$ defined over $\bbK$. Since $\Aut(\bar{X})$ contains 
$\Aut(X)$, $\bar{X}$
is Cremona special. 

>From our Main Theorem, we deduce that $\bar{X}$ is  unnodal, $\calP$ is made of $n$ distinct
points of $\bbP^2(\overline{\bbK})$, and $\bar{X}$ is a Halphen, a Coble, or a Harbourne example over $\bbK$. In the Halphen case, the genus $1$ fibration 
on $X$ is unique. In the Coble case, $\vert -2K_{\bar{X}}\vert$ contains a unique member (the strict transform of the 
sextic curve with double points along $\calP$). In the Harbourne case, when the Picard number of $\bar{X}$ is $\geq 10$, 
$\vert -K_{\bar{X}}\vert$ contains also a unique member (given by the proper transform of the cuspidal cubic containing
$\calP$). Thus, these curves and pencil are defined over $\bbK$.

Let us show that $\Pic(X)$ has finite index in $\Pic(\bar{X})$.
First, assume that $n\geq 10$. Consider the subgroup $\Pic(X)$ of $\Pic(\bar{X})$. 
It contains ample classes and the canonical class $K_X$; in particular, it intersects $K_X^\perp$ on an infinite subgroup~$L$
which is $\Aut(X)^*$-invariant. But $\Aut(X)^*$ has finite index in $W_{\bar{X}}$, $W_{\bar{X}}$ is Zariski 
dense in $\Or(K_X^\perp \otimes \bbR)$, and the action of this orthogonal group on $K_X^\perp\otimes \bbR$
is irreducible. Thus, $L\otimes \bbR$ coincides with $K_X^\perp\otimes \bbR$. 
Since  $\Pic(X)$ also contains $K_X$, we deduce that $\Pic(X)\otimes \bbR$
is equal to  $\Pic(\bar{X})\otimes \bbR$ and that $\Pic(X)$ has finite index in $\Pic(\bar{X})$. 

When $n=9$, one needs a slightly different argument.
Since  $\Pic(X)$ contains ample classes, the intersection form restricts to a form of signature $(1,m)$ on $\Pic(X)$, with $m\leq 9$. In particular, the intersection form is negative definite on the orthogonal complement of $\Pic(X)$, 
and the action of $\Aut(X)$ on $\Pic(X)$ has finite kernel. Since the action of 
$\Aut(X)$ on $\Pic(X)$ preserves the isotropic vector $K_X$ and the integral structure, 
it contains a finite index, free abelian subgroup of 
rank at most $m-1$. Since $\Aut(X)^*$ has finite index in $W_9$, we deduce that $m=9$ because $W_9$
contains a free abelian group of rank~$8$. This shows that $\Pic(X)$ has finite index in $\Pic(\bar{X})$. 

Let us now prove that all points $p_i$ of $\calP$, $1\leq i\leq n$, are in fact defined over $\mathbb{K}$,
i.e. that $\calP \subset \bbP^2(\mathbb{K})$.  It suffices to show that all $(-1)$-curves $E$ in $\bar{X}$ are defined over $\bbK$. Since $\Pic(X)$
has finite index in $\Pic(\bar{X})$, the divisor class of some positive multiple 
$mE$ is in $\Pic(X)$. Since $\vert mE\vert = \{mE\}$, we obtain  that $mE$ is defined over $\mathbb{K}$ and hence $E$ is defined over $\bbK$. 
\end{proof}

\subsection{Other types of surfaces} \label{par:NRS}
We can extend the concept of Cremona special rational surfaces to other types of projective surfaces as follows. 
One says that a surface $Y$ has a {\bf{large}} automorphism group if   
$\Aut(Y)^*$ is infinite and of finite index in the orthogonal 
group $\Or(K_Y^\perp) \subset \Or(\Num(Y))$, where $\Num(Y)$ is the lattice of divisor classes modulo numerical equivalence. Besides rational surfaces, other candidates of surfaces with large automorphism group are surfaces of Kodaira dimension 0 or 1. Indeed, $\Aut(Y)$ is finite if the Kodaira dimension of $Y$
is $2$, and $\Aut(Y)^*$ is finite if $Y$ is ruled but not rational.

\subsubsection{Kodaira dimension $1$} If the Kodaira dimension of $Y$ is 1, some multiple of the canonical class defines an elliptic  (or quasi-elliptic) fibration on $Y$.  The Mordell-Weil group of the corresponding jacobian fibration embeds as a finite index subgroup into $\Aut(Y)^*$. By Shioda-Tate formula, the rank of this group is equal to 
$\rho(Y)-2$, provided there are no reducible fibres in the genus 1 fibration. Using the argument from the proof  of Theorem \ref{Halphen-CS}, one can show, in this case, that $\Aut(Y)^*$ is of finite index in $\Or(K_Y^\perp)$. Thus,  {\sl{$\Aut(Y)^*$ has finite index in $\Or(K_Y^\perp)$ if and only if the canonical 
fibration has no reducible fiber}}.

\subsubsection{Kodaira dimension $0$} The classification of surfaces implies that minimal surfaces with Kodaira dimension $0$ fall into four types:  Abelian 
surfaces,  K3 surfaces, Enriques surfaces, and bi-elliptic surfaces.

\begin{theorem}\label{thm:Kod0}
Let $\bbK$ be an algebraically closed field with $\cha(\bbK)\neq 2$. Let $Y$ be a projective surface over $\bbK$ with Kodaira dimension equal to $0$. If $\Aut(Y)$
is large, then $Y$ is minimal and $Y$ is not a bi-elliptic surface. If $Y$ is an abelian surface, a K3 surface, or an Enriques surface, then $\Aut(Y)$ is large if and only if $\Or(\Num(Y))$ is infinite and $Y$ does not
contain any smooth rational curve. \end{theorem}

\begin{proof}[Sketch of the proof]
Recall that a $(-k)$-curve is a smooth rational curve with self-intersection $-k$.
Let $Y$ be a projective surface with Kodaira dimension $0$. The set of classes of $(-1)$-curves is finite, and permuted
by $\Aut(Y)$; thus, if $\Aut(Y)$ is large, $Y$ is minimal  and $K_Y^\perp = \Num(Y)$. If $Y$ is bi-elliptic, it is easily checked that  $\Aut(Y)^*$ is finite. 

Abelian surfaces do not contain rational curves, and smooth 
rational curves on K3 surfaces and Enriques surfaces are $(-2)$-curves. Each of them defines a reflection on $\Num(Y)$. Denote by ${\text{Nod}}(Y)$
the set of smooth rational curves   and by ${\text{Ref}}(Y)\subset \Or(\Num(Y))$ the group generated by the reflections around classes of smooth rational
curves. 

Let $Y$ be an abelian or K3 surface. Over the field of complex numbers, the Torelli Theorem implies that $\Aut(Y)^*$ is of finite index in $\Or(\Num(Y))$ if  and only if  ${\text{Nod}}(Y)$ is empty (this is always true for abelian surfaces). For K3 surfaces this fact has been extended recently  to any characteristic $p \ne 2$ by M. Lieblich and D. Maulik (see \cite{Lieblich-Maulik}). They also extended another corollary of the Torelli Theorem: The set of orbits of smooth rational curves with respect to the automorphism group is finite. Their arguments can  probably be adapted to abelian surfaces.  Thus,  an abelian surface or K3 surface has large automorphism group if and only if $\Or(\Num(Y))$ is infinite and ${\text{Nod}}(Y)$ is empty. \footnote{If $\Aut(Y)^*$ is infinite and $Y$ contains smooth rational curves, ${\text{Ref}}(Y)$ is infinite and its elements represent different cosets of $\Or(\Num(Y))$ modulo $\Aut(Y)^*$.}

Finally, the automorphism group of an Enriques surface without smooth rational curves is always large, and this is true in any characteristic. 
The proof and the statement are analogous to the  proof of  Theorem \ref{autocoble} (see \cite{Dolgachev:1984}, for $\bbK=\bbC$).  
\end{proof}

\subsubsection{Complex surfaces and rational curves} Let us assume that $Y$ is a complex projective surface. 
\begin{enumerate}
\item If $Y$ is an abelian surface, $Y$ does not contain rational curves;
\item if $Y$ is a K3 surface and $Y$ does not contain any smooth rational curve, its Picard number satisfies $\rho(Y)\leq 11$;
\item a generic Enriques surface contains no smooth rational curve.
\end{enumerate}
The first assertion is easily proved, and the third is contained in \cite{Barth-Peters:1983}. The second one was explained to the authors by V. Nikulin: {\sl{Any K3 surface with $\rho(Y) \geq 12$  contains a smooth rational curve}} (see \cite{Nikulin:2012}, Theorems 14 and 15). 

To prove it, assume $\rho(Y)\geq 12$. Since $H^2(Y,Z)$   is an even unimodular lattice of dimension $22$ and signature $(3,19)$, the orthogonal complement $T(Y)$
of $\Num(Y)$ in $H^2(Y,Z)$ has dimension $\leq 10$. By  Theorem 1.13.1* (or 1.13.2) in \cite{Nikulin:1979}, 
 the primitive embedding $T(Y)\subset H^2(Y,Z)$ is unique 
up to isomorphism, because $\Num(U)$ is indefinite and $\rank(T(Y))\le 10$ (see also Corollaries 2.9 and 2.10 in \cite{Morrison:1984}). 
On the other hand, since $\rank (T(Y))\le 10$, Theorem 1.12.2 of \cite{Nikulin:1979} implies the existence of
a primitive embedding 
\[
T(Y)\oplus \langle -2 \rangle \subset H^2(Y,Z)
\] 
where $\langle -2 \rangle$ denotes the $1$-dimensional lattice generated by a vector with self-intersection $-2$.
Taking orthogonal complement for this second embedding, one obtains that $\Num(Y)$ contains a copy of the sublattice $\langle -2 \rangle$.  

This shows that K3 surfaces with large automorphism groups  must  satisfy $\rho(Y) \le 11$. This is a strange coincidence with our main 
result that complex rational surfaces with large automorphism group must satisfy $\rho = 10$, or $11$. 

\begin{corollary}
Let $Y$ be a complex projective surface with large automorphism group. If $\Aut(Y)^*$ does not contain any finite index abelian group, 
the Picard number of  $Y$ is at most $11$. 
\end{corollary}

Indeed, the assumption implies that $\Aut(Y)^*$ is infinite but does not preserve a genus $1$ fibration, since otherwise the Mordell-Weil group
of the corresponding jacobian fibration would determine a finite index abelian subgroup of $\Aut(Y)^*$. Thus, either $Y$ is rational, or its
Kodaira dimension vanishes. The conclusion follows from the Main Theorem, \ref{thm:Kod0}, and Nikulin's argument.

\bibliographystyle{plain}
\bibliography{referencescoble}

\end{document}